\documentclass[a4paper,10pt]{amsart}
 
\usepackage[english]{babel}
\usepackage{dsfont} 
\usepackage{graphicx}
\usepackage{color}
\usepackage{enumerate}
\usepackage{array,multirow}

\allowdisplaybreaks

\usepackage{amssymb} 
\usepackage{amsthm} 
\usepackage{dsfont}

\newtheorem{thm}{Theorem}
\newtheorem{dfn}[thm]{Definition}
\newtheorem{lem}[thm]{Lemma}

\newtheorem{exm}[thm]{Example}

\newtheorem{prop}[thm]{Proposition}
\newtheorem{rem}[thm]{Remark}
\newtheorem{cor}[thm]{Corollary}

\newtheorem{open}{Problem}

\newcommand{\nset}{\mathds{N}}

\newcommand{\rset}{\mathds{R}}

\newcommand{\pset}{\mathds{P}}

\newcommand{\aff}{\mathrm{aff}\,}

\newcommand{\conv}{\mathrm{conv}\,}
\newcommand{\cone}{\mathrm{cone}\,}

\newcommand{\lin}{\mathrm{lin}\,}
\newcommand{\Pos}{\mathrm{Pos}}
\newcommand{\range}{\mathrm{range}\,}
\newcommand{\rank}{\mathrm{rank}\,}
\newcommand{\sign}{\mathrm{sgn}}

\newcommand{\supp}{\mathrm{supp}\,}
\newcommand{\und}{\;\wedge\;}
\newcommand{\folgt}{\;\Rightarrow\;}
\newcommand{\gdw}{\;\Leftrightarrow\;}
\newcommand{\inter}{\mathrm{int}\,}
\newcommand{\diff}{\mathrm{d}}
\newcommand{\pos}{\mathrm{Pos}}
\newcommand{\nor}{\mathrm{Nor}}
\newcommand{\relint}{\mathrm{relint}\,}
\newcommand{\one}{\mathds{1}}

\newcommand{\cA}{\mathcal{A}}
\newcommand{\cB}{\mathcal{B}}
\newcommand{\cat}{\mathcal{C}}
\newcommand{\cD}{\mathcal{D}}
\newcommand{\cE}{\mathcal{E}}
\newcommand{\cF}{\mathcal{F}}
\newcommand{\cK}{\mathcal{K}}
\newcommand{\cN}{\mathcal{N}}
\newcommand{\pyt}{\mathcal{P}}
\newcommand{\cS}{\mathcal{S}}

\newcommand{\cX}{\mathcal{X}}
\newcommand{\cW}{\mathcal{W}}
\newcommand{\cV}{\mathcal{V}}
\newcommand{\cZ}{\mathcal{Z}}
\newcommand{\cL}{\mathcal{L}}

\newcommand{\sA}{\mathsf{A}}
\newcommand{\sB}{\mathsf{B}}

\newcommand{\fA}{\mathfrak{A}}

\newcommand{\fM}{\mathfrak{M}}
\newcommand{\fR}{\mathfrak{R}}
\newcommand{\fp}{\mathfrak{p}}
\newcommand{\fq}{\mathfrak{q}}
\newcommand{\fw}{\mathfrak{w}}
\newcommand{\fz}{\mathfrak{z}}

\author{Philipp J.\ di~Dio}
\address{University of Leipzig, Mathematical Institute, Augustusplatz 10/11, D-04109 Leipzig, Germany}
\address{Max Planck Institute for Mathematics in the Sciences, Inselstra{\ss}e 22, D-04103 Leipzig, Germany}
\email{didio@uni-leipzig.de}

\author{Konrad Schm\"udgen}
\address{University of Leipzig, Mathematical Institute, Augustusplatz 10/11, D-04109 Leipzig, Germany}
\email{schmuedgen@math.uni-leipzig.de}

\begin{title}[The multidimensional truncated Moment Problem: The Moment Cone]{The multidimensional truncated Moment Problem: The Moment Cone}
\end{title}
\date{\today}

\begin{document}

\maketitle


\begin{abstract}
Let $\sA=\{a_1,\dots,a_m\}$, $m\in\nset$, be measurable functions on a measurable space $(\cX,\fA)$.
If $\mu$ is a positive measure on $(\cX,\fA)$ such that $\int a_i d\mu<\infty$ for all $i$, then the sequence $(\int a_1 d\mu,\dots,\int a_m d\mu)$ is called a moment sequence. By Richter's Theorem each moment sequence has a $k$-atomic representing measure with $k\leq m$. The set $\cS_\sA$ of all moment sequences is the moment cone. 

The aim of this paper is to analyze the various structures of   the moment cone. The main results concern  the facial structure (exposed faces, facial dimensions) and lower and upper bounds of the  Carath\'eodory number (that is, the smallest number of atoms which suffices for all moment sequences) of the convex cone $\cS_\sA$. In the case when $\cX\subseteq \rset^n$ and  $a_i\in C^1(\cX,\rset)$, the differential structure of the moment map and regularity/singularity properties of moment sequences are analyzed. The maximal mass problem is considered and some  applications to other problems are sketched. 
\end{abstract}



\tableofcontents

\section{Introduction}

The moment problem was first formulated and investigated by T.\ Stieltjes in his famous memoir \cite{stielt94}. However, the use of moments for the study of integrals of functions can be traced back to P.\ L.\ Chebyshev and A.\ A.\ Markov \cite{kreinMarkovMomentProblem}, and even to C.\ F.\ Gau{\ss} \cite{gauss15}. The moment problem is now a well-studied classical problem which has deep connections with many mathematical fields (see e.g.\ \cite{landauMomAMSProc}, \cite{lauren09}, \cite{fialkow10}, \cite{bernar13}, and \cite{schmudMomentBook}) and a broad scope of applications (see e.g.\ \cite{lasserreSemiAlgOpt}). In its most general form the moment problem is the following:

\textbf{Generalized Moment Problem:}
\textit{Let $(\cX,\fA)$ be a measurable space and $\sA = \{a_i\}_{i\in I}$ be a (finite or infinite) set of measurable (real-valued) functions on $(\cX,\fA)$. Given a (real) sequence $s = (s_i)_{i\in I}$, does there exists a (positive) measure $\mu$ on $(\cX,\fA)$ such that
\begin{equation}\label{eq:momentIntegral}
s_i = \int_\cX a_i(x)~\diff\mu(x) \quad\forall i\in I?
\end{equation}}

This general problem has a number of natural and important  subproblems:
\begin{enumerate}[1)]
\item Existence: When does there exist a measure $\mu$ such that (\ref{eq:momentIntegral}) holds?
\item Determinacy: Is the measure $\mu$ satisfying (\ref{eq:momentIntegral}) unique?
\item Characterization of solutions: In the case when the measure $\mu$ is not unique, how can one describe all solutions?
\item Simplification: Are there  natural classes of ``simplest'' solutions?
\end{enumerate}

The name \emph{moment} problem stems from the classical case of monomials on $\rset$: The number $s_n = \int_\rset x^n~\diff\mu(x)$ is called the \emph{$n$-th moment} of the measure $\mu$ on $\rset$. If $\sA$ is the set of all momomials $x^\alpha$ for $\alpha \in \nset_0^d$, then the problem is called  \emph{multi-dimensional full} moment problem. If $\sA$ consists of all monomials $x^\alpha$ up to a fixed degree $|\alpha|\leq m$, or more generally, if the set $\sA = \{a_1,\dots,a_m\}$ is finite,  we obtain the \emph{truncated} moment problem, which is the subject of this paper.  
Even if many results are formulated for general functions, the case of monomials $x^\alpha, |\alpha|\leq m$, is our main guiding example. By an important theorem of J.\ Stochel \cite{stochel01}, solving all truncated power moment problems implies solving the full power moment problem.

There is a huge literature of thousands research papers about versions and applications of moment problems. Even for the truncated moment problem there are an extensive literature and interactions with various fields in pure and applied mathematics, see \cite{ahiezer62}, \cite{landauMomAMSProc}, \cite{lauren09}, \cite{lasserreSemiAlgOpt}, and \cite{schmudMomentBook}. Let us mention some of  the main lines of such  interactions. As in the case of the full moment problem, there is a close interplay with real algebraic geometry, especially with nonnegative polynomials, see e.g.\ \cite{robinson69}, \cite{choi80}, \cite{reznick92}, \cite{matzkePhD}, \cite{harris99}, and \cite{marshallPosPoly}. Semi-algebraic sets and convex hulls of curves occur in a natural manner: The moment cone is semi-algebraic (Example \ref{exm:semialgSA}), the set of atoms is a real algebraic set (the core variety), and the moment cone is the convex hull of the moment curve. There are connections with sums of squares and Waring decompositions, see e.g.\ \cite{reznick92}, \cite{lauren09}, \cite{blekh12}, \cite{blekhe15a}, \cite{blekhe15}. An important method of solving truncated moment problems is based on flat extensions of Hankel  matrices, see e.g.\ \cite{curto3}, \cite{curto2}, \cite{lauren09a}, \cite{vasiles12}, \cite{curto13}, \cite{mourrain16}, and \cite{sauerArXiv2018}. Another related topic is polynomial optimization \cite{lasserreSemiAlgOpt}. Determining the maximal mass at a given atom leads to a convex optimization problem \cite{schmud15}, \cite[Section 18.4]{schmudMomentBook}. By Richter's Theorem, each truncated moment problem has a finitely atomic solution. This builds the bridge to the theory of numerical integration and  to   cubature formulas for the approximation of integrals, see e.g.\ \cite{stroud71}, \cite{moller76}, \cite{davis84},  \cite{putina97}, \cite{putina00}, \cite{fialkow10}, \cite{vasile14}.  A related problem deals with the smallest number of atoms for a moment sequence, the so-called Carath\'eodory number \cite{rienerOptima}, \cite{didio17Cara}. The  moment sequences and the  nonnegative polynomials on a set form cones in finite-dimensional real vector spaces, so convex hulls of curves \cite{scheider18} enter the study of truncated moment problems in a natural manner. This was noted already early in \cite{kemper68}, \cite{kemper87} and developed in \cite{didio17Cara} and \cite[Chapter 18]{schmudMomentBook}. Most questions concerning solutions of the truncated moment problem depends in fact on the structure of the moment cone. In our opinion, the moment cone is the fundamental object of this theory.

This paper is about the \emph{moment cone of the truncated moment problem} for a finite set $\sA=\{a_1,\dots$, $a_m\}$ of measurable functions on a measurable space $(\cX,\fA)$. Since each moment sequence has a finitely atomic representing measure, the corresponding moment functional is a positive linear combination of point evaluations. This in turn suggests other interesting subproblems such as:
\begin{enumerate}[1)]\setcounter{enumi}{4}
\item Number of atoms: How many atoms are needed to represent a given  moment sequence? What is the smallest number $n$ such that all moment sequences can be represented by $k\leq n$ atoms?  

\item Position  of atoms: Which points $x\in \cX$ can appear as atoms of a representing measure of  a given moment sequence?
\end{enumerate}

Question 5 leads to the Carath\'eodory number of a moment sequence and of the whole moment cone, while the answer to question 6 is given by the core variety.

This paper continues our study of truncated moment problems in \cite{schmud15}, \cite{didio17w+v+} and \cite{didio17Cara}. As noted above, the central topic of the present paper is the cone of all moment sequences, the so-called \emph{moment cone} $\cS_\sA$ . This is a convex but not necessarily closed cone in $\rset^m$. It spans the vector space  $\rset^m$ if  the set $\sA$ is linearly independent. Our main aim is to analyze the various structures of the moment cone. These are 
\begin{itemize}
\item the structure of the convex set (boundary, interior, exposed faces),
\item the differential structure (given by the derivative of the moment map if the functions $a_i$ are differentiable),
\item  the internal structure (regular and singular points).
\end{itemize}

We try to develop these topics and  properties of the moment cone as complete as possible (to the best of our knowledge) and hope that  the paper is  also readable by non-experts. In some sense,  this paper is a mixture of a survey of known facts and of a research paper with new results. To make the presentation self-contained, we occasionally repeat results and examples from our previous papers, sometimes with new proofs based on arguments from convex analysis. At the end of the paper a few open problems are listed. We hope that our treatment of the moment cone  encourages further research on this topic.

Let us briefly describe  the structure of this paper. In Section \ref{sec:momcone} we introduce the moment cone  $\cS_\sA$.  A crucial result is  Richter's Theorem which states that each moment sequence has a $k$-atomic representing measure with $k\leq m$. In Section \ref{sec:richter} we reprove this theorem  and present a detailed  historical discussion around this result. In Section \ref{sec:mommap} we define the moment map and analyze its differential structure in the case when $\cX\subseteq \rset^n$ and the functions $a_i$ are differentiable.

Sections \ref{sec:face} and \ref{sec:cara} are the main sections of this paper. In Section \ref{sec:face} we investigate the face structure and the geometry of the cone $\cS_\sA$, in particular the face $\cF_s$, the exposed face $\cE_s$ of a moment sequence $s$, and face dimensions. Section \ref{sec:cara} is about Carath\'eodory numbers. We present new lower and upper bounds based on face dimensions and show that $\cF_s$ and $\cE_s$ are closely related to the set of atoms $\cW(s)$ and the set $\cV(s)$. Using Harris' polynomial we develop in the polynomial case without gaps the first example of a moment sequence $s$ for which $\cV(s) \neq \cW(s)$.

In Section \ref{sec:inner} we study the interior of the moment cone and the regularity and singularity  of moment sequences. Section \ref{sec:appl} contains various applications (SOS and tensor decomposition, Pythagoras numbers for polynomials in one variable with gaps). Section \ref{sec:maxmass} deals with the maximal mass  and its relation to conic optimization. The maximal mass of a moment sequence $s$ at $x$ is equal to the infimum  of $L_s(p)$ over $p\geq 0$ with $p(x)=1$. Using again Harris' polynomial  we give an example where the infimum  is not a minimum. The final Section \ref{sec:open} collects some open problems which might be  starting points for future investigations.

\section{Preliminaries on Integration}

We briefly summarize some facts and notations from integration theory. For a deeper treatment we refer to standard texts such as \cite{bogachevMeasureTheory}.

Throughout, $(\cX,\fA)$ denotes a \emph{measurable space}. This means that $\fA$ is a $\sigma$-algebra on a non-empty set $\cX$. $\mu$ is a positive measure on $\fA$, i.e., $\mu: \fA\rightarrow [0,\infty]$, if it is a countably additive set function on $\fA$. For any $x\in \cX$, the Dirac (delta) measure
\[\delta_x(A) := \begin{cases} 1 & \text{if}\ x\in A,\\ 0 & \text{if}\ x\not\in A\end{cases}\]
is a measure on $(\cX,\fA)$; this holds  even when the one point set $\{x\}$ is not in $\fA$.

A measure $\mu$ is called \emph{$k$-atomic} if there are $k$ pairwise different points $x_1,\dots,x_k\in \cX$ and positive numbers $c_1,\dots,c_k$ such that 
\begin{equation}\label{eq:katomicmeasure}
\mu=\sum_{j=1}^k c_j\cdot \delta_{x_j}.
\end{equation}
The points $x_i$ are called \emph{atoms} and the numbers $c_i$ \emph{masses} of $\mu$. If the numbers $c_i$ in (\ref{eq:katomicmeasure}) are real rather than positive, then $\mu$ is called a \emph{$k$-atomic signed measure}.

A function $a:\cX\rightarrow\rset$ is \emph{$\fA$-measurable} if $a^{-1}(B):=\{x\in \cX:a(x)\in B\}\in\fA$ for all Borel subsets $B$ of $\rset$ and $a$ is called  $\mu$-integrable if
\begin{equation}\label{eq:athMoment}
\int_\cX a(x)~\diff\mu(x)\in (-\infty,\infty).
\end{equation}
By the definition of the Lebesgue integral, (\ref{eq:athMoment}) is equivalent to 
\begin{subequations}
\begin{align}
\Leftrightarrow\;\; &\int |a(x)|~\diff\mu(x)<\infty\label{eq:cond2}\\
\Leftrightarrow\;\; &\int a_+(x)~\diff\mu(x),\ \int a_-(x)~\diff\mu(x) < \infty,
\end{align}
\end{subequations}
where $a_+(x) := \max\{a(x),0\}$ and $a_-(x) := \max\{-a(x),0\}$, i.e., $a = a_+ - a_-$ and $|a| = a_+ + a_-$.

Throughout, $\sA = \{a_1,\dots,a_m\}$ denotes a finite set of real-valued measurable functions on $(\cX,\fA)$.

\begin{rem}\label{rem:finiteais}
In measure theory,   measurable functions can have values in $[-\infty,\infty]$. Since all moments have to be  finite, the functions $a_i$'s are integrable, so it is sufficient to work on the measurable space
\[\tilde{\cX} := \{x\in\cX \;|\; a_i(x)\in \rset\ \forall i=1,\dots,m\}.\]
On the measurable space $(\tilde{\cX},\fA|_{\tilde{\cX}})$ all functions $a_i$ have finite  real values. We will assume this in what follows.
\end{rem}

\begin{dfn}
\[\fM_\sA := \{\mu\ \text{positive measure on}\ (\cX,\fA) \,|\, a_1,\dots,a_m\ \text{are}\ \mu\text{-integrable}\}.\]
\end{dfn}

Since the functions $a_i$ are finite, $\delta_x\in\fM_\sA$ for any $x\in \cX$, so $\fM_\sA$ is not empty.

\begin{lem}\label{lem:diracmeasure}
$\delta_x\in\fM_\sA$ for all $x\in\cX$.
\end{lem}

It is eminent that the following holds:
\begin{equation}\label{eq:aposimpliesIntegralPos}
a(x) \geq 0\ \text{on}\ \cX \qquad\Rightarrow\qquad \int a(x)~\diff\mu(x) \geq 0\qquad \forall\; \mu\in\fM_\sA.
\end{equation}

\begin{exm}\label{exm:measurablespace}
Let $\cX = [0,2]$ and $\fA = \{\emptyset,[0,1),[1,2],\cX\}$. Then all measurable functions on $(\cX,\fA)$ are of the form $a(x) = c_1 \chi_{[0,1)}(x) + c_2 \chi_{[1,2]}(x)$, where  $\chi_A$ denotes the characteristic function on $A\in\fA$.  Then $\delta_x\in\fM_\sA$ for $x\in [0,2]$ despite the fact that $\{x\}\not\in\fA$ and we have  $\delta_{x_1} = \delta_{x_2}$ for all $x_1,x_2\in [0,1)$ or $x_1,x_2\in [1,2]$.
\end{exm}

As mentioned in the introduction, our guiding example will be the polynomial case and we therefore introduce the following notations: For $n,d\in\nset$ we set
\[\sA_{n,d} := \{x^\alpha = x_1^{\alpha_1}\cdots x_n^{\alpha_n} \;|\; |\alpha| = \alpha_1 + \dots + \alpha_n \leq d\} \quad\text{on}\quad \cX=\rset^n,\]
\[\sB_{n,d} := \{x^\alpha = x_0^{\alpha_0}\cdots x_n^{\alpha_n} \;|\; |\alpha| = \alpha_0 + \dots + \alpha_n = d\} \quad\text{on}\quad \cX=\pset^n,\]
$\cA_{n,d}:=\lin\sA_{n,d}$, and $\cB_{n,d}:=\lin\sB_{n,d}$. We have of course $|\sA_{n,d}| = |\sB_{n,d}| = \begin{pmatrix} n+d\\ n\end{pmatrix}$.

\section{The Moment Cone}
\label{sec:momcone}

The following definition collects a number of basic notions on truncated moment problems. Lemma \ref{lem:diracmeasure} motivates the following definition.

\begin{dfn}
The \emph{moment curve $s_\sA$} is defined by
\[s_\sA: \cX\rightarrow\rset^m,\ x\mapsto s_\sA(x) := \begin{pmatrix}a_1(x)\\ \vdots\\ a_m(x)\end{pmatrix}\]
and the \emph{moment cone $\cS_\sA$} is
\[\cS_\sA := \left\{ \int s_\sA(x)~\diff\mu(x)\,\middle|\, \mu\in\fM_\sA \right\}\subseteq\rset^m.\]
A sequence $s\in\cS_\sA$ is called a \emph{moment sequence} and a measure $\mu\in\fM_\sA$ with $s = \int s_\sA(x)~\diff\mu(x)$ is called a \emph{representing measure of $s$}. We denote the \emph{set of representing measures of $s$} by:
\[\fM_\sA(s) := \left\{ \mu\in\fM_\sA \,\middle|\, s = \int s_\sA(x)~\diff\mu(x) \right\}.\]
\end{dfn}

From the preceding definition the following is obvious.

\begin{lem}
The moment cone $\cS_\sA$  is a convex cone in $\rset^m$.
\end{lem}

Since $s_\sA(x) = \int s_\sA(y)~\diff\delta_x(y)$, each vector $s_\sA(x)$ belongs to the moment cone  $\cS_\sA$ and the Delta measure $\delta_x$ is a  representing measure of $s_\sA(x)$.

\begin{dfn}~
$\pos(\sA) := \{p\in\lin\sA \,|\, p(x)\geq 0\ \forall x\in\cX\}.$
\end{dfn}
\begin{lem}\label{lem:linearindependentA}
The following are equivalent:
\begin{enumerate}[i)]
\item The convex cone spanned by $s_{\sA}(\cX)$ is $k$-dimensional.
\item $\cS_\sA$ is $k$-dimensional.
\item $\lin\sA$ is a $k$-dimensional vector space.
\item $\exists b_1,\dots,b_k\in\lin\sA$ linearly independent with $\lin\sA = \lin\{b_1,\dots,b_k\}$.
\end{enumerate}
\end{lem}
\begin{proof}
iv) $\gdw$ iii) $\folgt$ i) $\folgt$ ii) are clear.

ii) $\folgt$ iii): Assume to the contrary that $\lin\sA$ is $l$-dimensional with $l < k$, then $\dim\cS_\sA \leq l < k$, which  is a contradiction.
\end{proof}

Throughout the rest of this paper we assume the following: 
\[\text{\bf The set of functions}\ \sA = \{a_1,\dots,a_m\}\ \text{\bf is linearly independent.}\tag{$*$}\]

By Lemma \ref{lem:linearindependentA} this is no loss of generality. It should be emphasized that basic properties  of the moment cone (facial structure, Carath\'eodory number etc.) do not depend on the particular choice of the linearly independent set $\sA$.

By $(*)$,  $\dim \lin \sA=m$. For each real sequence $s=(s_1,\dots,s_m)$ there is a unique   linear functional $L_s$ on $\lin \sA$,  called the {\it Riesz functional} of $s$, defined by $$L_s(a_i)=s_i,~~ i=1,\dots,m.$$  Obviously, if $s\in \cS_\cA$ and $\mu\in \fM_\sA(s)$, then $L_s(p)=\int p(x)\, d\mu$ for\, $p\in \lin \sA$.

\begin{dfn}
$\pos(\sA) := \{p\in\lin\sA \,|\, p(x)\geq 0\ \forall x\in\cX\}.$
\end{dfn}

Let $\cL_\sA$ denote the set of all Riesz functionals $L_s$ for $s\in \cS_\sA$. Clearly,  $\cL_\sA$ and
\[\pos(\sA)^\wedge=\{L\in (\lin\sA)^*: L(p)\geq 0\ \forall p\in \pos(\sA)\}\]
are cones in the dual space $(\lin\sA)^*$ and $\cS_\sA \cong \cL_\sA \subseteq\pos(\sA)^\wedge$.

\begin{exm}
Let $\sA=\{1,x,x^2\}$ on  $\cX=\rset$ and $s=(0,0,1)$. Then $L_s\in \pos (\sA)^\wedge$, but $L_s\notin \cL_\sA$ (i.e., $s\notin\cS_\sA$). (Indeed, otherwise $s_0=0$ and $\mu\in \fM_\sA(s)$ would imply $\mu=0$, which contradicts $s_2=1$.) 
\end{exm}

The dual cone $\pos(\sA)^\wedge$ coincides with  the closure of $\cL_\sA$ in  $(\lin\sA)^*$. Sufficient conditions for $\cL_\sA$ being closed,  equivalently   $\cL_\sA=\pos(\sA)^\wedge$, are given in \cite{schmudMomentBook}, see Theorems 1.26, 1.30, and  Proposition 1.27 therein.

Since $\cS_\sA$ is a convex cone, we repeat basic definitions and facts from the theory of convex sets and fix some notation, see e.g.\ \cite{rock72}, \cite{solta15}, or \cite{schne14} for more details.

For a $v\in\rset$ we set
\[H_v := \{x\in\rset^m \,|\, \langle v, x\rangle = 0\}\quad {\rm and}\quad  H_v^+ := \{x\in\rset^m \,|\, \langle v,x\rangle \geq 0\}.\]
If $v\neq 0$, then $H_v$ is called the \emph{hyperplane} with normal vector $v$ and $H_v^+$ the corresponding \emph{halfspace}. We shall say that $H_v^+$ is a \emph{containing halfspace} for $\cS_\sA$ iff $\cS_\sA \subseteq H_v^+$, that $H_v$ is a \emph{supporting hyperplane} iff $\cS_\sA\subseteq H_v^+$ and $\overline{\cS_\sA}\cap H_v\neq \emptyset$ and finally that $H_v$ \emph{supports $\cS_\sA$ at $s\in\cS_\sA$} iff $H_v$ is a supporting hyperplane and $s\in H_v$. For a convex cone $K\subseteq\rset^m$ we denote by $K^* := \{v\in\rset^m \,|\, K\subseteq H_v^+\}$ its \emph{dual cone}. An \emph{extreme face} (briefly, a face) of a convex set $K$ is a subset $F$ such that\  $\lambda x + (1-\lambda) y\in F$ for some $x,y\in K$ and $\lambda\in [0,1]$ implies $x,y\in F$. An \emph{exposed face} $F$ is a face of $K$ such that\ there exists a hyperplane $H$ with $F = K\cap H$. A cone $K$ is called \emph{line-free} (or \emph{pointed}) iff $K$ does not contain a line $x+y\cdot\rset$ with $x\in\rset^m$ and $y\in\rset^m\setminus\{0\}$. For $s\in\cS_\sA$ we define the \emph{normal cone} $\nor_\sA(s)$ by
\[\nor_\sA(s) := \{v\in\rset^m \,|\, H_v\ \text{is a supporting hyperplane of}\ \cS_\sA\ \text{at}\ s\}.\]
It is a closed convex cone, see e.g.\ \cite{solta15}. Finally, since our moment cone $\cS_\sA$ is not necessarily closed, it is necessary to distinguish between the boundary $\partial\cS_\sA$ of $\cS_\sA$ and the boundary points $\partial^*\cS_\sA$ belonging to $\cS_\sA$, that is, $\partial^*\cS_\sA := \partial\cS_\sA\cap \cS_\sA$.
\smallskip

The following  lemma  states that any proper exposed face ($\neq\cS_\sA$) of the moment cone $\cS_\sA$ is again a  moment cone $\cS_\sB$ on a subset $\cZ\subsetneq\cX$ and a basis $\sB=\{b_1,\dots,b_k\}\subset\lin\sA$ with $k<m$.

\begin{lem}\label{lem:boundaryreduction}
Let $H_v$ be a supporting hyperplane of $\cS_\sA$. Then $\cS_\sA\cap H_v$ is a  moment cone on $(\cZ,\fA|_\cZ)$ with $\cZ := \cZ(\langle v,s_\sA(\,\cdot\,)\rangle)$ and $\dim\cS_\sA\cap H_v < \dim\cS_\sA$.
\end{lem}
\begin{proof}
It is clear that $\dim\cS_\sA\cap H_v < \dim\cS_\sA$. By a basis change in the vector space $\lin \sA$  we can assume without loss of generality that $v=(0,\dots,0,1)$, so that $a_m \geq 0$. Therefore, $s=(s_1,\dots,s_m)\in\cS_\sA\cap H_v$ implies $s_m=0$ and
\[0 = s_m = \int a_m(y)~\diff\mu(y) \quad\forall\mu\in\fM_\sA(s) \quad\forall s\in\cS_\sA\cap H_v.\]
Then $\cZ = \cZ(a_m)$ and $\fA|_\cZ = \{M\cap\cZ \,|\, M\in\fA\}$. 
For any $x\not\in\cZ$ we have $a_m(x)\neq 0$, so $s_m\neq 0$, and $s_\sA(x) = s\not\in\cS_\sA\cap H_v$. Hence $\cS_\sA\cap H_v$ is the  moment cone on $(\cZ,\fA\cap\cZ)$.
\end{proof}

One might also ask which sequences are obtained by allowing signed atomic measures. The next proposition  (see \cite[p.\ 1608]{didio17Cara}) says that each vector $s\in\rset^m$ is a signed moment sequence and that a representing signed measure can easily be calculated by using  linear algebra.

\begin{prop}\label{prop:signedMeasures}
There exist points $x_1,\dots,x_m\in\cX$ such that\ every vector $s\in\rset^m$ has a $k$-atomic signed representing measure with $k\leq m$ and all atoms are from the set $\{x_1,\dots,x_m\}$, that is, $s\in\rset^m$ is  a signed moment sequence.
\end{prop}
\begin{proof}
Since the set $\sA$ of functions is linearly independent by the above assumption, there are points $x_1,\dots,x_m\in\cX$ such that\ the matrix $(s_\sA(x_1),\dots,s_\sA(x_m))\in\rset^{m\times m}$ has full rank. Therefore, for any $s\in\rset^m$, we have
\[s = c_1\cdot s_\sA(x_1) + \dots + c_m\cdot s_\sA(x_m) = \int s_\sA(y)~\diff\left(\sum_{i=1}^m c_i\cdot \delta_{x_i}\right)\]
with $(c_1,\dots,c_m)^T = (s_\sA(x_1),\dots,s_\sA(x_m))^{-1} s$.
\end{proof}

Thus, the crucial and hard part of the truncated moment problem is to find {\it positive} representing measures.

\section{The Moment Map and its Differential Structure}
\label{sec:mommap}

\begin{dfn}\label{dfn:momentCurveMap}
For  $k\in\nset$  the \emph{moment map $S_{k,\sA}$} is defined by
\[S_{k,\sA}: \rset_{\geq 0}^k\times\cX^k\rightarrow\rset^m,\ (C,X)\mapsto S_{k,\sA}(C,X) := \sum_{i=1}^k c_i\cdot s_\sA(x_i)\]
with $C = (c_1,\dots,c_k)$ and $X = (x_1,\dots,x_k)$.
\end{dfn}

More explicitly, the moment map $S_{k,\sA}$ has the form
\begin{equation}\label{eq:SkAmap}
S_{k,\sA}(C,X) = \begin{pmatrix}
c_1 a_1(x_1) + c_2 a_1(x_2) + \dots + c_k a_1(x_k)\\
\vdots\\
c_1 a_m(x_1) + c_2 a_m(x_2) + \dots + c_k a_m(x_k)
\end{pmatrix}\in\rset^m.
\end{equation}
The moment map $S_{k,\sA}$ is differentiable in all coefficients $c_i$. From (\ref{eq:SkAmap}) we obtain
\[\partial_{c_i} S_{k,\sA}(C,X) = \begin{pmatrix} a_1(x_i)\\ \vdots\\ a_m(x_i)\end{pmatrix} = s_\sA(x_i).\]
If in addition $\cX\subseteq\rset^n$ is open and all functions $a_i$ are differentiable, that is, $\sA\subseteq C^1(\cX,\rset)$, then $S_{k,\sA}$ is also differentiable in all coordinates $x_{i,j}$ of any atom $x_i = (x_{i,1}$, $\dots$, $x_{i,n})$ and from (\ref{eq:SkAmap}) we get
\[\partial_{x_{i,j}} S_{k,\sA}(C,X) = \begin{pmatrix} c_i \partial_j a_1(x_i)\\ \vdots\\ c_i\partial_j a_m(x_i)\end{pmatrix} = c_i \partial_j s_\sA(x_i),\]
where $\partial_j$ denotes the partial derivative with respect to the $j$-th coordinate. Then $S_{k,\sA}$ is differentiable  with respect to all $c_i$ and $x_{i,j}$ and  the total derivative is
\begin{equation}\begin{split}\label{eq:totalderivative}
&\ DS_{k,\sA}(C,X)\\ 
=&\ (\partial_{c_1}S_{k,{\sA}},\partial_{x_{1,1}}S_{k,{\sA}},\dots,\partial_{x_{1,n}}S_{k,{\sA}},\partial_{c_2}S_{k,{\sA}},\dots,\partial_{x_{k,n}}S_{k,{\sA}})\\
=&\ (s_{\sA}(x_1), c_1\partial_1 s_{\sA}|_{x=x_1},\dots,c_1\partial_n s_{\sA}|_{x=x_1},s_{\sA}(x_2),\dots,c_k \partial_n s_{\sA}|_{x=x_k})\\
=&\
\begin{pmatrix}
a_1(x_1) & c_1 \partial_1 a_1(x_1) & \dots & c_1 \partial_n a_1(x_1) & a_1(x_2) & \dots & c_k \partial_n a_1(x_k)\\
\vdots & \vdots & & \vdots & \vdots & & \vdots & \\
a_m(x_1) & c_1 \partial_1 a_m(x_1) & \dots & c_1 \partial_n a_m(x_1) & a_m(x_2) & \dots & c_k \partial_n a_m(x_k)
\end{pmatrix}.
\end{split}\end{equation}

Here we ordered the variables as $c_1,x_{1,1},\dots,x_{1,n},c_2,x_{2,1},\dots,x_{k,n}$, i.e., the first column is the derivative with respect to $c_1$, followed by the derivatives with respect to the coordinates $x_{1,1},\dots,x_{1,n}$ of the $1$st atom $x_1$, and so on. Since $DS_{k,\sA}$ is a matrix, we can calculate the rank of $DS_{k,\sA}(C,X)$ at some  atomic measure $(C,X)$.  As $k$ increases, there exists a $k\in \nset$ such that $DS_{k,\sA}(C,X)$ has rank $m$ at some $(C,X)$. This leads to the  definition of  the following important number.

\begin{dfn}\label{dfn:NA}
Suppose $\cX\subseteq\rset^n$ is open and $a_i\in C^1(\cX,\rset)$ for all $i$. We define
\[\cN_\sA := \min \{k\in\nset \,|\, \exists (C,X)\in\rset^k\times\cX^k: \rank DS_{k,\sA}(C,X) = m\}.\]
\end{dfn}

\begin{rem}\label{rem:manifold}
Instead of an open subset  $\cX$ of $\rset^n$, this definition and the following considerations can be extended almost verbatim to differentiable manifolds $\cX$. By choosing a chart $\varphi: U\subseteq\rset^n \rightarrow \cX$ of the manifold, $U$ open, the preceding definition makes sense for $\sA\circ\varphi = \{a_i\circ\varphi \,|\, i=1,\dots,m\}$. 
\end{rem}

The number $\cN_\sA$ is well-defined, because the total derivative $DS_{k,\sA}(C,X)$ contains the column vectors $s_\sA(x_1)$, $\dots$, $s_\sA(x_k)$. Since the functions $a_i$ are assumed to be linearly independent, we can find $x_1,\dots,x_m\in\cX$ s.t.\ $s_\sA(x_1),\dots,s_\sA(x_m)$ are linearly independent, i.e., $\cN_\sA \leq m$. Of course, the rank of $DS_{k,\sA}(C,X)$ does not depend on the numbers $c_i\neq 0$. Thus, we can set without loss of generality $c_i=1$ for all $i$, i.e., $C = (1,\dots,1) =:\one_k = \one$. We have the following lower bound.

\begin{prop}[{\cite[Prop.\ 23]{didio17Cara}}]
$\left\lceil\frac{m}{n+1}\right\rceil \leq \cN_\sA$.
\end{prop}

We illustrate the preceding by an example (taken from {\cite[Exm.\ 50]{didio17Cara}}).

\begin{exm}\label{exm:simpleexampleforNAanduniqueness}
Let $\sA = \sA_{2,2}$. Then
\[DS_{2,\sA}(C,X) = \begin{pmatrix}
1  & 0 & 0 & 1 & 0 & 0\\
x_{1,1} & c_1 & 0 & x_{2,1} & c_2 & 0\\
x_{1,2} & 0 & c_1 & x_{2,2} & 0 & c_2\\
(x_{1,1})^2 & 2 c_1 x_{1,1} & 0 & (x_{2,1})^2 & 2 c_2 x_{2,1} & 0\\
x_{1,1} x_{1,2} & c_1 x_{1,2} & c_1 x_{1,1} & x_{2,1} x_{2,2} & c_2 x_{2,2} & c_2 x_{2,1}\\
(x_{1,2})^2  & 0 & 2 c_1 x_{1,2} &(x_{2,2})^2  & 0 & 2 c_2 x_{2,2}
\end{pmatrix},\]
where $C=(c_1,c_2)$ and $X=(x_1,x_2)$,  $x_i = (x_{i,1},x_{i,2})$. From this we find that
\[\ker DS_{2,\sA}(C,X) = \rset\cdot 
v(C,X) \qquad\text{with}\qquad v(C,X):=\begin{pmatrix}
-2\\
c_1^{-1}(x_{1,1}-x_{2,1})\\
c_1^{-1}(x_{1,2}-x_{2,2})\\
2\\
c_2^{-1}(x_{1,1}-x_{2,1})\\
c_2^{-1}(x_{1,2}-x_{2,2})
\end{pmatrix}.\]
Hence  $\rank DS_{2,\sA_{2,2}} = 5$ for each point $(x_1,x_2),$  $x_1\neq x_2$. In order to get full rank $6$ we need a third atom. Therefore, it follows that $\cN_\sA = 3$.
\end{exm}

Determining the number  $\cN_\sA$ for general functions requires explicit calculations. But in the polynomial case (that is, for $\sA = \sA_{n,d}$ or $\sB_{n,d}$ on $\rset^n$, $\pset^n$ or on an open subset) the Alexander--Hirschowitz Theorem \cite{alexa95} gives the following important result.

\begin{thm}[{\cite[Thm.\ 53]{didio17Cara}}]\label{thm:NApolynomial}
We have
\begin{equation}\label{eq:NAforAllHomPoly}
\cN_{\sA_{n,d}} = \left\lceil \frac{1}{n+1}\begin{pmatrix}n+d\\ n\end{pmatrix} \right\rceil,
\end{equation}
except for the following five cases:
\begin{enumerate}[i)]
\item $d=2$: $\cN_{\sA_{n,2}} = n+1$.
\item $n=4$, $d=3$: $\cN_{\sA_{4,3}} = 8$.
\item $n=2$, $d=4$: $\cN_{\sA_{2,4}} = 6$.
\item $n=3$, $d=4$: $\cN_{\sA_{3,4}} = 10$.
\item $n=d=4$: $\cN_{\sA_{4,4}} = 15$.
\end{enumerate}
\end{thm}

Example \ref{exm:simpleexampleforNAanduniqueness} is just one of the exceptional cases of the Alexander--Hirschowitz Theorem which shows that $\cN_\sA$ is not given by (\ref{eq:NAforAllHomPoly}) in general.

\section{Richter's Theorem}
\label{sec:richter}

The following theorem due to Hans Richter (1957) is probably the most important general result on truncated moment problems and it is the starting point for many problems as well. For this reason  we include the proof of Richter \cite{richte57} rewritten in terms of convex analysis, see also \cite[Thm.\ 1.24]{schmudMomentBook}.  By Remark \ref{rem:finiteais} we can assume that all measurable functions have finite values.

\begin{thm}[Richter  {\cite[Satz 4]{richte57}}]\label{thm:richter}
Let $\sA =\{a_1,\dots,a_m\}$ be (finite) measurable functions on a measurable space $(\cX,\fA)$. Then every moment sequence $s\in\cS_\sA$ has a $k$-atomic representing measure with at most $k\leq m$ atoms. Thus, 
\begin{equation}\label{eq:atomicrepresentation}
\cS_\sA = \range S_{m,\sA} = \conv \cone s_\sA(\cX).
\end{equation}
If $s$ is a boundary point of $\cS_\sA$, then $m-1$ atoms are sufficient.
\end{thm}
\begin{proof}
We proceed  by induction on $m$. By Lemma \ref{lem:linearindependentA} we can assume that the set $\sA$ is linearly independent.

For \underline{$m=1$}, let $s\in\cS_\sA\subseteq\rset$ and $s\neq 0$. Since $s = \int a_1(x)~\diff\mu(x)\neq 0$ there is because of (\ref{eq:aposimpliesIntegralPos}) an $x\in\cX$ such that $a_1(x)\neq 0$ and $\sign\, s = \sign\, a_1(x)$. Then $c := s\cdot a_1(x)^{-1} > 0$ and
\[s = s\cdot a_1(x)^{-1}\cdot a_1(x) = c\cdot a_1(x) = \int a_1(y)~\diff(c\cdot\delta_x)(y),\]
i.e., $c\cdot\delta_x$ is a 1-atomic representing measure of $s$.

Let  \underline{$m>1$}. Suppose    the assertion of the theorem holds for all $k=1,\dots,m-1$. By Lemma \ref{lem:linearindependentA} $\cS_\sA$ and $\range S_{m,\sA}$ have non-empty interior. First we show that $\inter\cS_\sA = \inter\range S_{m,\sA}$.

Assume to the contrary that $\inter\cS_\sA \neq \inter\range S_{m,\sA}$. Then, since  $\range S_{m,\sA} \subseteq \cS_\sA$ and both sets are convex cones, $\inter\cS_\sA\setminus\inter\range S_{m,\sA}$ has inner points. Let $s$ be such an inner point with representing measure $\mu$. Then there exists a separating linear functional $l(\,\cdot\,) = \langle \vec{l},\,\cdot\,\rangle$ such that\ $l(s) < 0$ and $l(t) > 0$ for all $t\in\inter\range S_{m,\sA}$, i.e., $0\leq l(s_\sA(x)) = \langle \vec{l}, s_\sA(x)\rangle =: a(x)$ for all $x\in\cX$. In other words,
\begin{equation}\label{eq:richtercontradiction}
a(x) \geq 0\ \text{on}\ \cX \;\;\text{but}\;\; \int a(x)~\diff\mu(x) = \int \langle \vec{l}, s_\sA(x)\rangle~\diff\mu(x) = \langle \vec{l}, s\rangle = l(s) < 0.
\end{equation}
This is a contradiction to (\ref{eq:aposimpliesIntegralPos}) and therefore we have $\inter\cS_\sA = \inter\range S_{m,\sA}$. This proves the assertion for inner points $s$ of $\cS_\sA$.

If $s$ is a boundary point of the moment cone $\cS_\sA$, $s$ is contained in an exposed face of the cone $\cS_\sA$. By Lemma \ref{lem:boundaryreduction}, this is again a moment cone of dimension $k\leq m-1$ and the assertion follows from the induction hypothesis.
\end{proof}

Thus, by Theorem \ref{thm:richter}, any moment sequence $s\in\cS_\sA$ can be written as
\begin{equation}\label{repspoints}
s = \sum_{i=1}^k c_i\cdot s_\sA(x_i) = \int_\cX s_\sA(x)~\diff\!\!\left(\sum_{i=1}^k c_i\cdot\delta_{x_i}\right)\!\!(x)
\end{equation}
for some points $x_1,\dots,x_k\in\cX$, positive real numbers $c_1,\dots,c_k$, and $k\leq m$.

Requiring that $(\cX,\fA)$ is a measurable space and $\sA$ are measurable functions is necessary to define integration. $\{x\}\in\fA$ is not required. On the other side, a number of generalizations can be made for $s_\sA(x)\in\rset^m$. We can replace $\rset^m$ by any finite dimensional vector space $V$ (for instance, matrices or tensors) and $\sA = \{a_1,\dots,a_m\}$ by the coordinate functions of any basis $v_1,\dots,v_m$: $v = a_1 v_1 + \dots + a_m v_m\in V$.

It should be emphasized that Theorem \ref{thm:richter} holds for arbitrary measurable spaces $(\cX,\fA)$ and measurable functions $\sA$, see Example \ref{exm:measurablespace} for $\{x\}\not\in\fA$ for all $x\in\cX$. We illustrate its use by another example.

\begin{exm} 
Let $s$ be a moment sequence on a measure space $(\cX,\fA)$ with representing measure $\mu$. If $M\in\fA$ such that $\mu(M)=0$, then it follows from Richter's Theorem (Theorem \ref{thm:richter}), applied to $\cX\setminus M$, that the points $x_i$ in (\ref{repspoints}) can be chosen from the set $\cX\backslash M$. For instance, suppose  $\mu$ is the Lebesgue measure of a closed set $\cX$ of $\rset^m$. If the Lebesgue measure of the boundary of $\cX$ is zero,  we can take the atoms $x_i$ from the interior of $\cX$. Also, if $M$ is  a finite subset of $\cX$, we  can choose the atoms outside $M$.
\end{exm}

The history of Richter's Theorem \ref{thm:richter} is confusing and intricating and often the corresponding references in the literature are misleading. For this reason we take this opportunity to discuss this history in detail. First we collect several versions of Theorem \ref{thm:richter} occurring in the literature in chronological order.
\begin{enumerate}[\bf A)]
\item \textbf{A.\ Wald 1939
\cite[Prop.\ 13]{wald39}:}
Let $\cX = \rset$ and $a_i(x) := |x - x_0|^{d_i}$ with $d_i\in\nset_0$ and $0\leq d_1 < d_2 < \dots < d_m < \infty$ for an $x_0\in\cX$. Then (\ref{eq:atomicrepresentation}) holds.

\item\label{item:rosenbloom} \textbf{P.\ C.\ Rosenbloom 1952 \cite[Cor.\ 38e]{rosenb52}:}
Let $(\cX,\fA)$ be a measurable space and $a_i$ bounded measurable functions. Then (\ref{eq:atomicrepresentation}) holds.

\item \textbf{H.\ Richter 1957\footnote{Received: December 27, 1956. Published: April, 1957.} \cite[Satz 4]{richte57}:}
Let $(\cX,\fA)$ be a measurable space and let $a_i$ be measurable functions. Then (\ref{eq:atomicrepresentation}) holds.

\item\label{item:tchakaloff} \textbf{M.\ V.\ Tchakaloff 1957\footnote{Published: July-September, 1957.} \cite{tchaka57}:}
Let $\cX\subset\rset^n$ be compact and $a_\alpha(x) = x^\alpha$, $|\alpha|\leq d$. Then (\ref{eq:atomicrepresentation}) holds.

\item \textbf{W.\ W.\ Rogosinski 1958\footnote{Received: August 22, 1957. Published: May 6, 1958.} \cite[Thm.\ 1]{rogosi58}:}
Let $(\cX,\fA)$ be a measurable space and let $a_i$ be measurable functions. Then (\ref{eq:atomicrepresentation}) holds.
\end{enumerate}
From this list we see that Tchakaloff's result \ref{item:tchakaloff}) from 1957 is a special case of Rosenbloom's result \ref{item:rosenbloom}) from 1952 and that the general case was proved by Richter and Rogosinski almost about at the same time, see the exact dates in the footnotes.  If one reads  Richter's paper, one might think at first glance that he treats only the one-dimensional case, but a closer look reveals that his Proposition (Satz) 4 covers actually the  general case of measurable functions. Rogosinski treats the one-dimensional case, but he also states that his proof works for general measurable spaces. The above proof of Theorem \ref{thm:richter}, and likewise the one in \cite[Thm.\ 1.24]{schmudMomentBook}, are nothing but modern formulations of the proofs of Richter and Rogosinski without additional arguments. Note that Rogosinki's paper \cite{rogosi58} was submitted about a half year after the appearance of Richter's \cite{richte57}.

It might be of interest to note that the general results of Richter and Rogosinski can be easily derived from Rosenbloom's Theorem by the following simple trick. Let $\sA = \{a_1,\dots,a_m\}$ be (finite) measurable functions on $(\cX,\fA)$ and set $\sB = \{b_1,\dots,b_m\}$, where $b_i := \frac{a_i}{f}$ with $f := 1 + \sum_{i=1}^m a_i^2$. Then
\begin{align*}
s\in\cS_\sB \;&\Leftrightarrow\; \exists\nu\in\fM_\sB: s = \int s_\sB(x)~\diff\nu(x) \gdw s = \int \frac{s_\sA(x)}{f(x)}~\diff\nu(x) = \int s_\sA(x)~\diff\mu(x)\\
&\Leftrightarrow\; \exists\mu\in\fM_\sA: s = \int s_\sA(x)~\diff\mu(x) \gdw s\in\cS_\sA \quad\text{with}\quad \diff\mu = f^{-1}~\diff\nu.
\end{align*}
Since all functions $b_i$ are bounded, Rosenbloom's Theorem applies to $\sB$, so each sequence $s\in\cS_\sB = \cS_\sA$ has a $k$-atomic representing measure $\nu\in\fM_\sB(s)$ with $k\leq m$ and scaling by $f^{-1}$ yields a $k$-atomic representing measure $\mu\in\fM_\sA(s)$:
\[s = \sum_{i=1}^k c_i\cdot s_\sB(x_i) = \sum_{i=1}^k \frac{c_i}{f(x_i)}\cdot s_\sA(x_i).\]

Theorem \ref{thm:richter} was overlooked  in the modern literature on truncated polynomial moment problems. It was  reproved in several papers in weaker forms and finally  in the polynomial case in \cite{bayer06}. But Theorem \ref{thm:richter} for general measurable functions was known and cited by J.\ H.\ B.\ Kemperman in \cite[Thm.\ 1]{kemper68} and attributed therein to Richter and Rogosinski. In the moment problem community succeeding Kemperman the general form of Theorem \ref{thm:richter} was often used, see e.g.\ \cite[eq.\ (2.3)]{kemper71}, \cite[p.\ 29]{kemper87}, \cite[Thm.\ 1, p.\ 198]{floudas01}, \cite[Thm.\ 1]{anast06}, and \cite[Thm.\ 2.50]{lasserreSemiAlgOpt}.

In \cite{schmudMomentBook} Theorem \ref{thm:richter} was called Richter--Tchakaloff Theorem.  After the preceding discussion and comparing the proofs and the precise  publication data we are convinced that it is  fully justified to use the name \emph{Richter Theorem}. If one wants to take the broader history of this result into account, it might be also fair and appropriate to call it  Richter--Rogosinski--Rosenbloom Theorem.

\section{Facial Structure of the Moment Cone $\cS_\sA$ and Set of Atoms}
\label{sec:face}

Let $s$ denote a fixed sequence of the moment cone $\cS_\sA$.

\begin{dfn}
The \emph{set of atoms} of $s$ is the set
\[\cW(s) := \{x\in\cX \,|\, \exists \mu\in\fM_\sA(s) : \mu(A_x) > 0\}\]
with $A_x := s_\sA^{-1}(s_\sA(x))\equiv \{y\in \cX \,|\, s_\sA(y)=s_\sA(x)\}$.
\end{dfn}

The name \emph{set of atoms} for $\cW(s)$ is justified by the following lemma.

\begin{lem}\label{lem:setofatoms}
\[x\in\cW(s) \quad\Leftrightarrow\quad s = c\cdot s_\sA(x) + \sum_{i=1}^k c_i\cdot s_\sA(x_i) \quad\text{with}\quad x_i\in\cX,\ c>0,c_i > 0.\]
\end{lem}
\begin{proof}
$\Leftarrow$: Clear.

$\Rightarrow$: Let $\mu\in\fM_\sA(s)$ such that $\mu(A_x) > 0$. Then
\[s = \int_\cX s_\sA(y)~\diff\mu(y) = \int_{A_x} s_\sA(y)~\diff\mu(y) + \int_{\cX\setminus A_x} s_\sA(y)~\diff\mu(y) = \mu(A_x) s_\sA(x) + s'\]
with $s' = \int_{\cX\setminus A_x} s_\sA(y)~\diff\mu(y)$. Hence $s'\in\cS_\sA$ with representing measure $\mu|_{\cX\setminus A_x}$. By Richter's Theorem (Theorem \ref{thm:richter}), $s'$ has  an atomic representing measure which gives the assertion.
\end{proof}

The next definition introduces two other fundamental notions, the cone $\cN(s)$  and the set $\cV(s)$.

\begin{dfn}\ \vspace{-0.55cm}
\begin{align*}
\cN(s) &:= \{ p\in\pos(\sA) \,|\, L_s(p) = \langle\vec{p},s\rangle = 0\},\\
\cV(s) &:= \bigcap_{p\in\cN(s)} \cZ(p).
\end{align*}
\end{dfn}

Basic properties of these concepts are collected in the next propositions. Note that in the important case when $\sA$ consists of polynomials and $\cX=\rset^d$, $\cV(s)$ is a real zero set of polynomials and hence a real algebraic set, see Proposition \ref{prop:onep}.

\begin{prop}
\begin{enumerate}[i)]
\item $p\in\pos(\sA) \gdw \cS_\sA\subseteq H_{\vec{p}}^+$.
\item $p\in\cN(s) \gdw \cS_\sA\subseteq H_{\vec{p}}^+ \und s\in H_{\vec{p}} \gdw \vec{p}\in\nor_\sA(s)$.
\item Let $p\in\pos(\sA)$. Then $L_s(p) = \langle \vec{p}, s\rangle= 0 \folgt \cW(s)\subseteq\cZ(p)$.
\item $\cW(s)\subseteq\cV(s)$.
\item There exists a  $p\in\pos(\sA)$ with $\cV(s) = \cZ(p)$.
\item Suppose that there exist a function $ e\in\lin\sA$ s.t.\ $e(x)>0$ on $\cX$. Then: $0\in\fM_\sA(s) \gdw s=0 \gdw \cW(s)=\emptyset$.
\end{enumerate}
Let $\cX$ be a locally compact Hausdorff space. Suppose that  $\sA$ is contained in  $ C(\cX,\rset)$ and there exists a function $e\in\lin\sA$ such that $e(x)>0$ on $\cX$. Then:
\begin{enumerate}[i)]\setcounter{enumi}{6}
\item Let $p\in\pos(\sA)$. If $L_s(p) = \langle\vec{p},s\rangle = 0$, then  $ \supp\mu\subseteq\cZ(p)~ \forall\mu\in\fM_\sA(s)$.
\item $s=0 \gdw \fM_\sA(s)=\{0\}$.
\end{enumerate}
\end{prop}
\begin{proof}
i): $p\in\pos(\sA) \gdw p(x)=\langle\vec{p},s_\sA(x)\rangle\geq 0\ \forall x\in\cX \gdw \langle\vec{p},c_1 s_\sA(x_1)+ \dots +c_m s_\sA(x_m)\rangle\ \forall x_1,\dots,x_m\in\cX$ and $c_1,\dots,c_m\in (0,\infty)\; \overset{(*)}{\Leftrightarrow}\; \langle \vec{p},s\rangle \geq 0\ \forall s\in\cS_\sA \gdw \cS_\sA\subseteq H_{\vec{p}}^+$, where ($*$) holds by Richter's Theorem \ref{thm:richter}.

ii): It is easily verified that all three assertions mean that $\langle\vec{p},s\rangle=0$ and $\cS_\sA\subseteq H_{\vec{p}}^+$. Hence they are equivalent.

iii): Let $x\in\cW(s)$. Then by Lemma \ref{lem:setofatoms} there are $c,c_i > 0$ and $x_i\in\cX$ such that\ $\mu = c\cdot\delta_x + \sum_{i=1}^k c_i \cdot s_\sA(x_i)\in\fM_\sA(s)$. Therefore, since $p\geq 0$ on $\cX$, we obtain
\[0 = L_s(p) = \int_\cX p(y)~\diff\mu(y) = c\cdot p(x) + \sum_{i=1}^k c_i\cdot p(x_i) \folgt x,x_i\in\cZ(p).\]

iv): Follows at once from iii).

v): \cite[Prop.\ 29]{didio17w+v+}.

vi): The first equivalence is clear and in the second the direction  $\Leftarrow$ follows from Richter's Theorem \ref{thm:richter}. For the remaining implication assume there is a $x\in\cW(s)$. Then, by  Lemma \ref{lem:setofatoms}, we have $s = c\cdot s_\sA(x) + \sum_{i=1}^k c_i\cdot s_\sA(x_i)$ and
\[0 = L_{s=0}(e) = c\cdot e(x)  + \sum_{i=1}^k c_i\cdot e(x_i) \geq c\cdot e(x) > 0,\]
which is a contradiction. Thus, $\cW(s) = \emptyset$.

vii): We repeat this well-known argument (see e.g. \cite[Proposition 1.23]{schmudMomentBook}). Suppose $x\notin \cZ(p)$. Since $p$ is continuous on $\cX$, there are an open neighborhood $U$ of $x$ and a constant $\delta >0$ such that $p(x)\geq \delta>0$ on $U$. Then
\[0=L_(p)=\int_\cX p(x) d\mu \geq \int_U p(x)\, d\mu \geq \delta \mu(U)\geq 0.\]
Hence $\mu(U)=0$, so that $x\notin \supp \mu$.

viii): Set $p=e$ in vi).
\end{proof}


\begin{prop}[{\cite[Thm.\ 4.6]{didioDiss}}]\label{prop:onep}
If $\vec{p}\in\relint\nor_\sA(s)$, then $\cV(s) = \cZ(p)$.
\end{prop}
\begin{proof}
Let $v\in\nor_\sA(s)$. Since $\vec{p}\in\relint\nor_\sA(s)$, there is a number $\varepsilon_v > 0$ such that\ $\vec{q}_v := \vec{p}-\varepsilon_v\cdot v\in\nor_\sA(s)$, i.e., $q_v := \langle\vec{q}_v,s_\sA(\,\cdot\,)\rangle\geq 0$. Then
\[\cZ(p) \subseteq \cZ(q_v + \varepsilon_v \langle v,s_\sA(\,\cdot\,)\rangle \subseteq \cZ(\langle v,s_\sA(\,\cdot\,)\rangle) \qquad\forall v\in\nor_\sA(s),\]
i.e.,
\[\cZ(p) \subseteq \bigcap_{v\in\nor_\sA(s)} \cZ(\langle v,s_\sA(\,\cdot\,)\rangle =: \cV(s) \subseteq \cZ(p).\qedhere\]
\end{proof}

The next theorem is \cite[Thm.\ 30]{didio17w+v+}. It is valid in the measurable case as well with verbatim the same proof using Lemma \ref{lem:setofatoms}.

\begin{thm}\label{thm:Vs=Ws}
$\cV(s) = \cW(s) \gdw s\in\relint F\ \text{for an exposed face}\ F\ \text{of}\ \cS_\sA$.
\end{thm}

By definition, the moment cone $\cS_\sA$ itself is also an exposed face and we have $\cW(s)=\cX$ for all $s\in\inter\cS_\sA$. Theorem \ref{thm:Vs=Ws} shows that exposed faces and faces $F$ of $\cS_\sA$ such that $s\in\relint F$ play a central role for the study of the set of atoms and the structure of $\cS_\sA$.

 S.\ Karlin and L.\ S.\ Shapley \cite{karlin53}  investigated  the face structure of the moment cone $\cS_\sA$ in the special case $\sA = \{1,x,\dots,x^d\}$ on $\cX = [0,1]$. We will generalize this study to the multivariate truncated moment problem.

\begin{dfn}
For $s\in\cS_\sA$ we define the \emph{face $\cF_s$ of $s$} as the face of $\cS_\sA$ such that $s\in\relint\cF_s$ and the \emph{dimension $\cD_s$ of $s$} as $\cD_s := \dim\cF_s$. Additionally, we define the \emph{exposed face $\cE_s$ of $s$} as the smallest exposed face of $\cS_\sA$ containing $s$.
\end{dfn}

It is clear from the definition that $\cF_s$ and $\cE_s$ are unique and $\cF_s \subseteq \cE_s$.  In \cite{karlin53} $\cE_s$ is called \emph{contact set} and $\cF_s$ is the \emph{reduced contact set}, since it is obtained by an iterated cutting  out of $\cS_\sA$, see e.g.\ Theorem \ref{thm:WsZps} and the discussion afterwards.

\begin{lem}
Any face $F$ of $\cS_\sA$ is of the form $\cF_s$ for some $s\in\cS_\sA$.
\end{lem}
\begin{proof}
Take an element $s\in\relint F$.
\end{proof}

The next proposition connects the set $\cW(s)$ with $\cF_s$ and $\cV(s)$ with $\cE_s$.

\begin{prop}
\begin{enumerate}[i)]
\item $s_\sA(x)\in\cF_s \gdw x\in\cW(s)$, i.e.,
\[\cW(s) = s_\sA^{-1}(\cF_s) \qquad\text{and}\qquad \cF_s = \conv\cone s_\sA(\cW(s)).\]

\item $s_\sA(x)\in\cE_s \gdw x\in\cV(s)$, i.e.,
\[\cV(s) = s_\sA^{-1}(\cE_s) \qquad\text{and}\qquad \cE_s = \conv\cone s_\sA(\cV(s)).\]

\item The sets $\cW(s)$ and $\cV(s)$ are measurable.
\end{enumerate}
\end{prop}
\begin{proof}
i): ``$\Rightarrow$'': Since $s\in\relint\cF_s$ and $s_\sA(x)\in\cF_\sA$ there is an $\varepsilon > 0$ such that\ $s' := s - \varepsilon\cdot s_\sA(x)\in\cF_s\subseteq\cS_\sA$. Let $\mu'\in\fM_\sA(s')$, then $\mu := \varepsilon\cdot\delta_x + \mu'\in\fM_\sA(s)$ and therefore $x\in\cW(s)$ by Lemma \ref{lem:setofatoms}.

``$\Leftarrow$'': Let $x\in\cW(s)$. Then, by Lemma \ref{lem:setofatoms}, $s = c\cdot s_\sA(x) + \sum_{i=1}^k c_i\cdot s_\sA(x_i)$ for some $c,c_i> 0$ and $x_i\in\cX$. But since $s_\sA(x),s_\sA(x_i)\in\cS_\sA$ and $\cF_s$ is a face of $\cS_\sA$ containing $s$, also $s_\sA(x),s_\sA(x_i)\in\cF_s$.

ii): Let $s'\in\relint\cE_s$ and apply i) with $\cW(s') = \cV(s)$.

iii): Since the functions of $\sA$ are measurable and $\sA$ is a finite set, $s_\sA$ is a measurable function and therefore $\cW(s) = s_\sA^{-1}(\cF_s)\in\fA$ is measurable by i). $\cV(s)$ is measurable by ii) or by $\cV(s) = \cZ(p)$ from Proposition \ref{prop:onep}.
\end{proof}

So $\cW(s)\subseteq\cV(s)$ means geometrically that $\cF_s\subseteq\cE_s$ and Theorem \ref{thm:Vs=Ws} means geometrically that $\cF_s = \cE_s$. From the preceding considerations we find the following useful description of the set of atoms $\cW(s)$.

\begin{thm}\label{thm:WsZps} 
For any $s\in\cS_\sA$  there exist functions $p\in\pos(\sA)$ (i.e., $p$ with $\cV(s)=\cZ(p)$ from Proposition \ref{prop:onep}) and $p_1,\dots,p_k\in\lin\sA$ with $k \leq m - \cD_s -1$ such that
\begin{equation}\label{eq:sWzeros}
\cW(s) = \cZ(p)\cap \cZ(p_1)\cap \dots\cap \cZ(p_k).
\end{equation}
\end{thm}
\begin{proof}
First let $s\in\inter\cS_\sA$. Then, as noted above, $\cW(s) = \cX$, so $\cW(s) = \cZ(p)$ with $p=0$. Now let $s\in\partial^*\cS_\sA$. Then there is a supporting hyperplane $H_v$ of $\cS_\sA$ such that\ $s\in H_v$ (w.l.o.g.\ $p$ s.t.\ $\cV(s)=\cZ(p)$ from Proposition \ref{prop:onep}). Additionally, there are $v_1,\dots,v_k\in\rset^m$ with $k\leq m-\cD_s-1$ such that\ $\lin\cF_s = \lin s_\sA(\cW(s)) = H_v\cap H_{v_1}\cap \dots\cap H_{v_k}$. Then
\begin{align*}
\cW(s) &= s_\sA^{-1}(\cF_s) = s_\sA^{-1}(\cS_\sA\cap\lin\cF_s) = s_\sA^{-1}(\cS_\sA\cap H_v\cap H_{v_1}\cap \dots\cap H_{v_k})\\ &= \cZ(p)\cap \cZ(p_1)\cap\dots\cap\cZ(p_k)
\end{align*}
with $p(x) := \langle v, s_\sA(x)\rangle$ and $p_i(x) := \langle v_i,s_\sA(x)\rangle$.
\end{proof}

W.l.o.g.\ we can assume $p_i\in\pos(\sA,\cZ(p)\cap\cZ(p_1)\cap\dots\cap\cZ(p_{i-1}))$ in the proof.

Theorem \ref{thm:WsZps} says that we can cut $\cF_s$ out of $\cS_\sA$ by finitely many hyperplanes. Note that for $s\in\inter\cS_\sA$ we have $\cW(s)=\cX$ by \cite[Thm.\ 16]{didio17Cara} and therefore $p=0$ and $k\leq -1$, i.e., no function $p_i$ is required. The inequality $k \leq m - \cD_s -1$ gives an upper bound for $k$, but  Example \ref{exm:harris} or Theorem \ref{thm:endingat2} show that it may happen that $k < m - \cD_s -1$.  For a fixed $\cS_\sA$ let $k_{\max}\leq m - 2$ be the maximal $k$ needed such that (\ref{eq:sWzeros}) holds for any $s\in\cS_\sA$. By taking an $s\in\relint\cS_\sA\cap H_{v_1}\cap\dots\cap H_{v_k'}$ we see for any $k'=0,\dots,k_{\max}$ there is an $s\in\cS_\sA$ such that $k'$ is the minimal $k$ with (\ref{eq:sWzeros}).

First we note that Richter's Theorem \ref{thm:richter} implies that the moment cone $\cS_\sA$ with $\sA = \{a_1,\dots,a_m\}\subset\rset[x_1,\dots,x_n]$ on a semi-algebraic $\cX\subseteq\rset^n$ and its dual cone $\pos(\sA)$ are semi-algebraic, i.e., there exist finitely many polynomial inequalities for the components $s_i$ of the sequence $s = (s_i)_{i=1}^m$ or the coefficients $c_i$ of the polynomial $p = \sum_{i=1}^m c_i a_i$ for deciding $s\in\cS_\sA$ or $p\in\pos(\sA)$.

\begin{exm}\label{exm:semialgSA}
Let $\sA = \{1,x,x^2\}$. Set $q(x) = 1+x^2$ and $\sB = \{1+x^2,x,x^2\}$, so $\cS_\sB = \conv \cone \range s_\sB(\rset)$, i.e.,
\begin{equation}\label{eq:base}
\conv\range \begin{pmatrix}\frac{x}{1+x^2}\\ \frac{x^2}{1+x^2}\end{pmatrix}
\end{equation}
is a base of the cone $\cS_\sB$. Note that (\ref{eq:base}) is a circle in $\rset^2$ with center $(0,1/2)$ and radius $1/2$ without the point $(0,1)$. The point $(0,1)$ corresponds to an ``atom at infinity''. The semi-algebraic description of (\ref{eq:base}) is
\[\{(x,y)\in\rset^2 \;|\; x^2 + (y-1/2)^2 \leq 1/4 \;\;\wedge\;\; y < 1\},\]
where $x = \frac{s_1}{s_0+s_2}$ and $y = \frac{s_2}{s_0+s_2}$. Then $s=(s_0,s_1,s_2)\in\cS_\sA\setminus\{0\}$ if and only if
\[s_0>0,\quad s_2\geq 0 \quad\text{and}\quad \left(\frac{s_1}{s_0+s_2}\right)^2 + \left(\frac{s_2}{s_0+s_2}-\frac{1}{2}\right)^2 \leq \frac{1}{4}.\]
\end{exm}

Of course, if $\cS_\sA$ is semi-algebraic, so are $\cE_s$ and $\cF_s$ for any $s\in\cS_\sA$. For polynomials on semi-algebraic sets Theorem \ref{thm:WsZps} has the following corollary.

\begin{cor}\label{cor:semialgWs}
Let $\cX$ be a (semi-)algebraic set  in $\rset^n$ and  $\sA$ a set of polynomials. Then the set of atoms $\cW(s)$ is (semi-)algebraic for any moment sequence $s\in\cS_\sA$. The same is true if $\cX$ is (semi-)algebraic in $\pset^n$ and $\sA$ consists of homogeneous polynomials.
\end{cor}

We return to the face structure of $\cS_\sA$.

\begin{dfn}\label{dfn:perfect}
A convex set $K\subseteq\rset^m$ is called \emph{perfect} iff every face of $K$ is also an exposed face.
\end{dfn}

An immediate consequence of Theorem \ref{thm:Vs=Ws} yields the following.

\begin{cor}
The following are equivalent:
\begin{enumerate}[i)]
\item The moment cone $\cS_\sA$ is perfect.
\item $\cV(s) = \cW(s)$ for all $s\in\cS_\sA$.
\item $\cF_s = \cE_s$ for all $s\in\cS_\sA$.
\end{enumerate} 
\end{cor}

In the one-dimensional monomial case $\sA = \{1,x,x^2,\dots,x^d\}$ on a closed interval of $\rset$ we always have $\cV(s) = \cW(s)$, see e.g.\ \cite[Exm.\ 14]{didio17w+v+}. The first examples for which $\cW(s)\neq\cV(s)$ were \cite[Examples\ 38 and 39]{didio17w+v+}, see also \cite[Example 18.25]{schmudMomentBook}; all known examples of this kind were one-dimensional monomials with gaps. The following is the first example for polynomials without gaps.

\begin{exm}[Harris polynomial $\cV(s)\neq\cW(s)$]\label{exm:harris}
W.\ R.\ Harris \cite{harris99} proved that the polynomial
\begin{align}
h(x_0,x_1,x_2) =\;\; 16&(x_0^{10} + x_1^{10} + x_2^{10})\notag\\
-36&(x_0^8 x_1^2 + x_0^2 x_1^8 + x_0^8 x_2^2 + x_0^2 x_2^8 + x_1^8 x_2^2 + x_1^2 x_2^8)\notag\\
+20&(x_0^6 x_1^4 + x_0^4 x_1^6 + x_0^6 x_2^4 + x_0^4 x_2^6 + x_1^6 x_2^4 + x_1^4 x_2^6)\label{eq:harrisPoly}\\
+57&(x_0^6 x_1^2 x_2^2 + x_0^2 x_1^6 x_2^2 + x_0^2 x_1^2 x_2^6)\notag\\
-38&(x_0^4 x_4^4 x_2^2 + x_0^4 x_1^2 x_2^4 + x_0^2 x_1^4 x_2^4)\notag
\end{align}
in $\cB_{2,10}$ is nonnegative on $\pset^2$ and has the projective zero set
\begin{align*}
\cZ(h) = \{&(1,1,0)^*,(1,1,\sqrt{2})^*,(1,1,1/2)^*\} \\
= \{&z_1=(1,1,0),z_2=(1,-1,0),z_3=(1,0,1),z_4=(1,0,-1),z_5=(0,1,1),\\ &z_6=(0,1,-1),\\
&z_{7}=(1,1,1/2),z_{8}=(1,1,-1/2),z_{9}=(1,-1,1/2),z_{10}=(1,-1,-1/2),\\
&z_{11}=(1,1/2,1),z_{12}=(1,1/2,-1),z_{13}=(1,-1/2,1),z_{14}=(1,-1/2,-1),\\
&z_{15}=(1/2,1,1),z_{16}=(1/2,1,-1),z_{17}=(1/2,-1,1),z_{18}=(1/2,-1,-1),\\
&z_{19}=(1,1,\sqrt{2}),z_{20}=(1,1,-\sqrt{2}),z_{21}=(1,-1,\sqrt{2}),z_{22}=(1,-1,-\sqrt{2}),\\
&z_{23}=(1,\sqrt{2},1),z_{24}=(1,\sqrt{2},-1),z_{25}=(1,-\sqrt{2},1),z_{26}=(1,-\sqrt{2},-1),\\
&z_{27}=(\sqrt{2},1,1),z_{28}=(\sqrt{2},1,-1),z_{29}=(\sqrt{2},-1,1),z_{30}=(\sqrt{2},-1,-1)\}.
\end{align*}
Here the symbol $(a,b,c)^*$ denotes all permutations of $(a,b,c)$ including sign changes. Hence, $h$ has exactly $30$ projective zeros. Set $Z_k := \{1,\dots,z_k\}$. Note that the full rank of $DS_{k,\sB_{2,10}}$ is $|\sB_{2,10}| = 66$. Table \ref{tab:indexExHarris} shows the rank of $DS_{k,\sA}(\one,Z_k)$ for $Z_k$.

\begin{table}[h!]
\caption{Rank of $DS_{k,\sA}(\one,Z_k)$ of subsets $Z_k=\{z_1,\dots,z_k\}$ of the zero set of the Harris polynomial. $|\sB_{2,10}|=66$.}\label{tab:indexExHarris}\centering\vspace{0.2cm}
\begin{tabular}{c|cccccccccc}\hline\hline
            & $Z_1$ & $Z_2$ & $Z_3$ & $Z_4$ & $Z_5$ & $Z_6$ & $Z_7$ & $Z_8$ & $Z_9$ & $Z_{10}$\\ \hline
$\rank DS_{k,\sA}(\one,Z_k)$ & $3$     & $6$     & $9$     & $12$    & $15$    & $18$    & $21$    & $24$    & $27$    & $30$\\
increase    & $+3$  & $+3$  & $+3$  & $+3$  & $+3$  & $+3$  & $+3$  & $+3$  & $+2$  & $+1$\\ \hline\hline
            & $Z_{11}$ & $Z_{12}$ & $Z_{13}$ & $Z_{14}$ & $Z_{15}$ & $Z_{16}$ & $Z_{17}$ & $Z_{18}$ & $Z_{19}$ & $Z_{20}$\\ \hline
$\rank DS_{k,\sA}(\one,Z_k)$ & $33$       & $36$       & $39$       & $42$       & $45$       & $48$       & $51$       & $54$       & $57$     & $60$\\
increase    & $+3$     & $+3$     & $+3$     & $+3$     & $+3$     & $+3$     & $+3$     & $+3$     & $+3$   & $+3$\\ \hline\hline
            & $Z_{21}$ & $Z_{22}$ & $Z_{23}$ & $Z_{24}$ & $Z_{25}$ & $Z_{26}$ & $Z_{27}$ & $Z_{28}$ & $Z_{29}$ & $Z_{30}$\\ \hline
$\rank DS_{k,\sA}(\one,Z_k)$ & $62$       & $63$       & $65$     & $65$    & $65$    & $65$    & $65$    & $65$    & $65$    & $65$\\
increase    & $+2$     & $+1$     & $+2$   & $+0$  & $+0$  & $+0$  & $+0$  & $+0$  & $+0$  & $+0$\\ \hline\hline
\end{tabular}
\end{table}

In \cite[Exm.\ 63]{didio17Cara} we showed by using a computer algebra program that the set $\{s_{\sB_{2,10}}(z_i)\}_{i=1}^{30}$ is linearly independent, i.e., the exposed face of $h$ has dimension 30: $\dim \cS_{\sB_{2,10}}\cap H_{\vec{h}} = 30$.  Table \ref{tab:indexExHarris} shows that already the zeros $z_1,\dots,z_{23}$ give $\rank DS_{23,\sB_{2,10}}(\one,Z_{23})=65$. Hence each sequence $s = \sum_{i=1}^{23} c_i\cdot s_{\sB_{2,10}}(z_i)$ with $c_i>0$ gives $\cV(s)=\cZ(h)=\{z_1,\dots,z_{30}\}$. But  $\cW(s)=\{z_1,\dots,z_{23}\}$, so that $\cV(s)\neq \cW(s)$.

Since $\{s_{\sB_{2,10}}(z_i)\}_{i=1}^{30}$ is linearly independent, for any $I\subset \cZ(h)$  there is a polynomial $p_I\in\lin\sB_{2,10}$ such that\ $p_I(z)=0$ for all $z\in I$ and $p_I(z)\neq 0$ for  $z\in J:=\cZ(h)\setminus I$. Therefore, $\cW(s) = \cZ(h)\cap\cZ(p_{\cW(s)})$ in Theorem \ref{thm:WsZps}, i.e., $k = 1 < m - \cD_s -1 = 34$.
\end{exm}

An immediate consequence of Example \ref{exm:harris} is the following.

\begin{cor}\label{cor:a10b10}
\begin{enumerate}[a)]
\item $\cS_{\sA_{1,d}}$ and $\cS_{\sB_{1,2d}}$ are perfect for all $d\geq 1$.
\item $\cS_{\sA_{n,10}}$ and $\cS_{\sB_{n,10}}$ for $n\geq 2$ are not perfect.
\end{enumerate}
\end{cor}

That $\cS_{\sA_{n,10}}$ and $\cS_{\sB_{n,10}}$ are not perfect indicates that also $\cS_{\sA_{n,d}}$ and $\cS_{\sB_{n,2d}}$ for $n\geq 2$ and $d\geq 1$ are not perfect (see Problem \ref{open:perfect}).

\begin{prop}\label{thm:endingat2}
If $s\in\cS_\sA$ such that the set\ $\cV(s)$ is finite, then 
\[\cW(s) = \cZ(p) \quad\text{or}\quad \cW(s) = \cZ(p)\cap\cZ(q)\]
for some $p\in\pos(\sA)$ and $q\in\lin\sA$ indefinite.
\end{prop}
\begin{proof}
Set $\cS := \conv \cone s_\sA(\cV(s))$. Then $\cS_\sA$ is a simplicial cone, i.e., every extreme face is an exposed face and the assertion follows from Theorem \ref{thm:Vs=Ws}.
\end{proof}

The conclusion of the previous theorem does not hold if $\cV(s)$ is infinite.

As shown by Example \ref{exm:harris}, the set $\cW(s)$ of atoms can be  smaller than   $\cV(s)$. To repair this the procedure  of defining  $\cV(s)$ will be iterated. 

Let $L$ be a linear functional on $\lin \sA$.
We define inductively linear subspaces $\cN_k(L)$, $k\in \nset,$ of $\lin \sA$ and  subsets $\cV_j(L)$, $j\in \nset_0,$ of $\cX$ by $V_0(L)=\cX$, 
\begin{align*}
\cN_k(L)&:=\{ p\in \lin \sA: L(p)=0,~~ p(x)\geq 0 ~~{\rm for}~~x\in \cV_{k-1}(L)\, \},\\
\cV_j(L)&:=\{ x\in \cX: p(x)=0~~{\rm for}~p\in \cN_j(L)\}.
\end{align*}If $\cV_k(L)$ is empty for some $k$, we set $\cV_j(L)=\emptyset$ for  $j\geq k$. 
Clearly, if $s\in \cS_\sA$, then $\cV_1(L_s)=\cV(s)$. From this definition it is obvious that $\cV_j(L)\subseteq \cV_{j-1} (L)$ for $j\in \nset.$

\begin{dfn}\label{corevarietyarb} 
The \emph{core variety} $\cV_C(L)$ of the linear functional $L$ on $\lin {\sA}$ is 
\begin{align}\label{corevarA} 
\cV_C(L):=\bigcap_{j=0}^\infty \cV_j(L).
\end{align}
\end{dfn}

The core variety was introduced by L. Fialkow \cite{fialkow17} and studied in \cite{didio17w+v+}, \cite{schmudMomentBook},  \cite{blek18coreArxiv}. If $\sA$ consists of real polynomials and $\cX=\rset^d$, then $\cV_C(L)$ is  the zero set of real polynomials and hence a real algebraic set in $\rset^d$.

Let $L_s$ be the Riesz functional of a moment sequence $s\in \cS_\sA$. Then, as proved in \cite[Theorem 33]{didio17w+v+}, see also \cite[Theorem 1.49]{schmudMomentBook}, we have $\cV_C(L_s)=\cW(s)$, that is, the core variety coincides with the set of atoms, see the most general form Theorem \ref{thm:WsZps}. Further, there exists a $k\in \nset_0$ such that $\cV_C(L_s)=\cV_k(L_s)$. The core variety is treated in \cite[Sections 1.2.5,  18.3]{schmudMomentBook}.

\smallskip

Let us resume the investigation of the face structure of the moment cone.

\begin{dfn}
For  $s\in\cS_\sA$ we define
\[\Gamma_s := \{f\in\lin\sA \,|\, \cW(s)\subseteq \cZ(f) \} \qquad\text{and}\qquad \gamma_s := \dim\Gamma_s.\]
\end{dfn}

\begin{prop}
$\cD_s = |\sA| - \gamma_s$.
\end{prop}
\begin{proof}
Clearly, $f\in\Gamma_s$ if and only if $\cF_s\subseteq \aff\cF_s \subseteq H_{\vec{f}}$. This gives the assertion.
\end{proof}

\begin{prop}\label{prop:DsBoundBoundary}
Let $\sA = \sA_{n,d}$ on $\cX=\rset^n$ or $\sA=\sB_{n,2d}$ on $\cX=\pset^n$. If $s\in\partial^*\cS_\sA$, then
\[\cD_s \leq |\sA| - n-1.\]
\end{prop}
\begin{proof}
Since $s\in\partial^*\cS_\sA$, there is a  $p\in\pos(\sA)\setminus\{0\}$ such that\ $\cW(s)\subseteq\cZ(p)$. Since $p\neq 0$, there is a $i$ such that\ $\partial_i p \neq 0$ and the functions $\partial_i p$ (resp. $x_0\partial_i p$ in the projective case), $x_1 \partial_i p,\dots,x_n\partial_i p$ are non-zero, linearly independent, and in $\Gamma_s$. Hence, $\gamma_s \geq n+1$ and $\cD_s\leq |\sA| - \gamma_s \leq |\sA| - n-1$.
\end{proof}

The preceding shows that the face $\cF_s$ of $s$ and its dimension $\cD_s$ play an important role  in the study of the moment cone $\cS_\sA$ and its boundary. Proposition \ref{prop:DsBoundBoundary} contains an upper bound for $\cD_s$. To give some lower bounds we consider two special cases:
\begin{enumerate}
\item[a)] $\cX = \rset^n$ (or $\pset^n$) with $\sA = \sA_{n,2d}$ (or $\sB_{n,2d}$, respectively). Set
\begin{equation}\label{eq:p}
\fp := p_1 + \dots + p_n \qquad\text{with}\qquad p_i(x) := \prod_{j=0}^{d-1} (x_i-j)^2,
\end{equation}
i.e., $\fp\geq 0$ on $\rset^n$ and $\cZ(\fp) = \{0,1,\dots,d-1\}^n$ with $\#\cZ(\fp) = d^n$.

\item[b)] $\cX = [0,d]^n$ with $\sA = \sA_{n,2d}$. Set
\begin{equation}\label{eq:q}
\fq := q_1 + \dots + q_n \qquad\text{with}\qquad q_i(x) := x\cdot \prod_{j=1}^{d-1} (x_i - j)^2 \cdot (d-x),
\end{equation}
i.e., $\fq \geq 0$ on $[0,d]^n$ and $\cZ(\fq) = \{0,1,\dots,d\}^n$ with $\#\cZ(\fq) = (d+1)^n$.
\end{enumerate}
Both cases are the simplest cases of non-negative polynomials with large numbers, but  finitely many zeros. The first case works on the whole space $\rset^n$ (or $\pset^n$). The second case on $[0,d]^n$ is important in numerical analysis. Since $\cZ(\fp)$ and $\cZ(\fq)$ are finite, we can set $s := \sum_{x\in\cZ(\fp)} s_\sA(x)$ and $s': = \sum_{x\in\cZ(\fq)} s_\sA(x)$. Then 
\[\cD_s = \fR_{n,2d} := \rank (s_{\sA_{n,2d}}(x))_{x\in\cZ(\fp)}\]
and
\[\cD_{s'} = \fR'_{n,2d} := \rank (s_{\sA_{n,2d}}(x))_{x\in\cZ(\fq)}.\]
In addition, we set $\fw_{n,2d} := \frac{\fR_{n,2d}}{\dim\cS_\sA}$ and $\fz_{n,2d} := \frac{\fR_{n,2d}}{|\cZ(\fp)|}$. The numbers $\fw'_{n,2d}$ and $\fz'_{n,2d}$ are defined in the same way for the second case b). By these definitions, $\fw_{n,2d}$ and $\fz_{n,2d}$ are  the ratios of the dimension of the exposed face $\cF_s$ by the dimension of $\dim\cS_\sA = |\sA_{n,2d}| = \left(\begin{smallmatrix} n+2d\\ n\end{smallmatrix}\right)$ and the cardinality of the zero set $\cZ(\fp)$, respectively.

For $n=1$ we can use the formula for the Vandermonde determinant and  obtain the following.

\begin{lem}
\begin{enumerate}[i)]
\item $\fR_{1,2d} = d$, $\fw_{1,2d} = \frac{d}{2d+1}$, and $\fz_{1,2d} = 1$.
\item $\fR'_{1,2d} = d+1$, $\fw'_{1,2d} = \frac{d+1}{2d+1}$, and $\fz'_{1,2d} = 1$.
\end{enumerate}
\end{lem}

For $n=2$, C.\ Riener and M.\ Schweighofer in \cite{rienerOptima} proved  the following result.

\begin{lem}[{\cite[Lem.\ 8.6]{rienerOptima}}]
$\fR_{2,2d} = d^2$, $\fw_{2,2d} = \frac{d^2}{(d+1)(2d+1)}$, and $\fz_{2,2d} = 1$.
\end{lem}

For $n \geq 3$ very little is known about these numbers. In table \ref{tab:Ds} we collect several numerical examples which have been calculated by a computer algebra program,  see also fig.\ \ref{fig:lower}.

\begin{table}[!p]\center\small
\caption{Values of $\cD_{s} = \fR_{n,2d}$, $\fw_{n,2d}$, and $\fz_{n,2d}$ for $\sA_{n,2d}$ on $\rset^n$; as well as $\cD_{s'} = \fR'_{n,2d}$, $\fw'_{n,2d}$, and $\fz'_{n,2d}$ for $\sA_{n,2d}$ on $[0,d]^n$ calculated by a computer algebra program.}\label{tab:Ds}
\begin{tabular}{rrr|rrrr|rrrr}
\hline\hline
$n$&$d$&$|\sA_{n,2d}|$ & $|\cZ(\fp)|$ & $\fR_{n,2d}$ & $\fw_{n,2d}$ & $\fz_{n,2d}$ & $|\cZ(\fq)|$ & $\fR'_{n,2d}$ & $\fw'_{n,2d}$ & $\fz'_{n,2d}$\\\hline
 3 & 1 &    10 &     1 &    1 & 10.0\% & 100.0\% &     8 &             7 &        70.0\% &        87.5\%\\
   & 2 &    35 &     8 &    8 & 22.9\% & 100.0\% &    27 &            23 &        65.7\% &        85.2\%\\
   & 3 &    84 &    27 &   27 & 32.1\% & 100.0\% &    64 &            54 &        64.3\% &        84.4\%\\
   & 4 &   165 &    64 &   63 & 38.2\% &  98.4\% &   125 &           105 &        63.6\% &        84.0\%\\
   & 5 &   286 &   125 &  121 & 42.3\% &  96.8\% &   216 &           181 &        63.3\% &        83.8\%\\
   & 6 &   455 &   216 &  206 & 45.3\% &  95.4\% &   343 &           287 &        63.1\% &        83.7\%\\
   & 7 &   680 &   343 &  323 & 47.5\% &  94.2\% &   512 &           428 &        62.9\% &        83.6\%\\
   & 8 &   969 &   512 &  477 & 49.2\% &  93.2\% &   729 &           609 &        62.8\% &        83.5\%\\
   & 9 &  1330 &   729 &  673 & 50.6\% &  92.3\% &  1000 &           835 &        62.8\% &        83.5\%\\
   &10 &  1771 &  1000 &  916 & 51.7\% &  91.6\% &  1331 &          1111 &        62.7\% &        83.5\%\\
   &11 &  2300 &  1331 & 1211 & 52.7\% &  91.0\% &  1728 &          1442 &        62.7\% &        83.4\%\\
   &12 &  2925 &  1728 & 1563 & 53.4\% &  90.5\% &  2197 &          1833 &        62.7\% &        83.4\%\\
   &13 &  3654 &  2197 & 1977 & 54.1\% &  90.0\% &  2744 &          2289 &        62.6\% &        83.4\%\\
   &14 &  4495 &  2744 & 2458 & 54.7\% &  89.6\% &  3375 &          2815 &        62.6\% &        83.4\%\\
   &15 &  5456 &  3375 & 3011 & 55.2\% &  89.2\% &  4096 &          3416 &        62.6\% &        83.4\%\\\hline
 4 & 1 &    15 &     1 &    1 &  6.6\% & 100.0\% &    16 &            11 &        73.3\% &        68.8\%\\
   & 2 &    70 &    16 &   16 & 22.9\% & 100.0\% &    81 &            50 &        71.4\% &        61.7\%\\
   & 3 &   210 &    81 &   76 & 36.2\% &  93.8\% &   256 &           150 &        71.4\% &        58.6\%\\
   & 4 &   495 &   256 &  221 & 44.6\% &  86.3\% &   625 &           355 &        71.7\% &        56.8\%\\
   & 5 &  1001 &   625 &  503 & 50.2\% &  80.5\% &  1296 &           721 &        72.0\% &        55.6\%\\
   & 6 &  1820 &  1296 &  986 & 54.2\% &  76.1\% &  2401 &          1316 &        72.3\% &        54.8\%\\
   & 7 &  3060 &  2401 & 1746 & 57.1\% &  72.7\% &  4096 &          2220 &        72.5\% &        54.2\%\\
   & 8 &  4845 &  4096 & 2871 & 59.3\% &  70.1\% &  6561 &          3525 &        72.8\% &        53.7\%\\\hline
 5 & 1 &    21 &     1 &    1 &  4.7\% & 100.0\% &    32 &            16 &        76.2\% &        50.0\%\\
   & 2 &   126 &    32 &   31 & 24.6\% &  96.9\% &   243 &            96 &        76.2\% &        39.5\%\\
   & 3 &   462 &   243 &  192 & 41.6\% &  79.0\% &  1024 &           357 &        77.3\% &        34.9\%\\
   & 4 &  1287 &  1024 &  667 & 51.8\% &  65.1\% &  3125 &          1007 &        78.2\% &        32.2\%\\
   & 5 &  3003 &  3125 & 1753 & 58.4\% &  56.1\% &  7776 &          2373 &        79.0\% &        30.5\%\\
   & 6 &  6188 &  7776 & 3888 & 62.8\% &  50.0\% & -- & -- & -- & --\\
   & 7 & 11628 & 16807 & 7678 & 66.0\% &  45.7\% & -- & -- & -- & --\\\hline
 6 & 1 &    28 &     1 &    1 &  3.6\% & 100.0\% &    64 &            22 &        78.6\% &        34.4\%\\
   & 2 &   210 &    64 &   57 & 27.1\% &  89.1\% &   729 &           168 &        80.0\% &        23.0\%\\
   & 3 &   924 &   729 &  435 & 47.1\% &  59.7\% &  4096 &           756 &        81.8\% &        18.5\%\\
   & 4 &  3003 &  4096 & 1758 & 58.5\% &  42.9\% & 15625 &          2499 &        83.2\% &        16.0\%\\
   & 5 &  8008 & 15625 & 5251 & 65.6\% &  33.6\% & -- & -- & -- & --\\\hline
 7 & 1 &    36 &     1 &    1 &  2.2\% & 100.0\% &   128 &            29 &        80.6\% &        22.7\%\\
   & 2 &   330 &   128 &   99 & 30.0\% &  77.3\% &  2187 &           274 &        83.0\% &        12.5\%\\
   & 3 &  1716 &  2187 &  897 & 52.3\% &  41.0\% & 16384 &          1464 &        85.3\% &         8.9\%\\
   & 4 &  6435 & 16384 & 4146 & 64.4\% &  25.3\% & -- & -- & -- & --\\\hline
 8 & 1 &    45 &     1 &    1 &  2.2\% & 100.0\% &   256 &            37 &        82.2\% &        14.5\%\\
   & 2 &   495 &   256 &  163 & 32.9\% &  63.7\% &  6561 &           423 &        85.5\% &         6.4\%\\
   & 3 &  3003 &  6561 & 1711 & 57.0\% &  26.1\% & -- & -- & -- & --\\\hline
 9 & 1 &    55 &     1 &    1 &  1.8\% & 100.0\% &   512 &            46 &        83.6\% &         9.0\%\\
   & 2 &   715 &   512 &  256 & 35.8\% &  50.0\% & 19683 &           625 &        87.4\% &         3.2\%\\
   & 3 &  5005 & 19683 & 3061 & 61.2\% &  15.6\% & -- & -- & -- & --\\\hline
10 & 1 &    66 &     1 &    1 &  1.5\% & 100.0\% &  1024 &            56 &        84.8\% &         5.5\%\\
   & 2 &  1001 &  1024 &  386 & 38.6\% &  37.7\% & 59049 &           891 &        89.0\% &   1.5\%\\\hline\hline
\end{tabular}
\end{table}

\begin{figure}[ht!]\centering
\includegraphics[width=0.75\columnwidth]{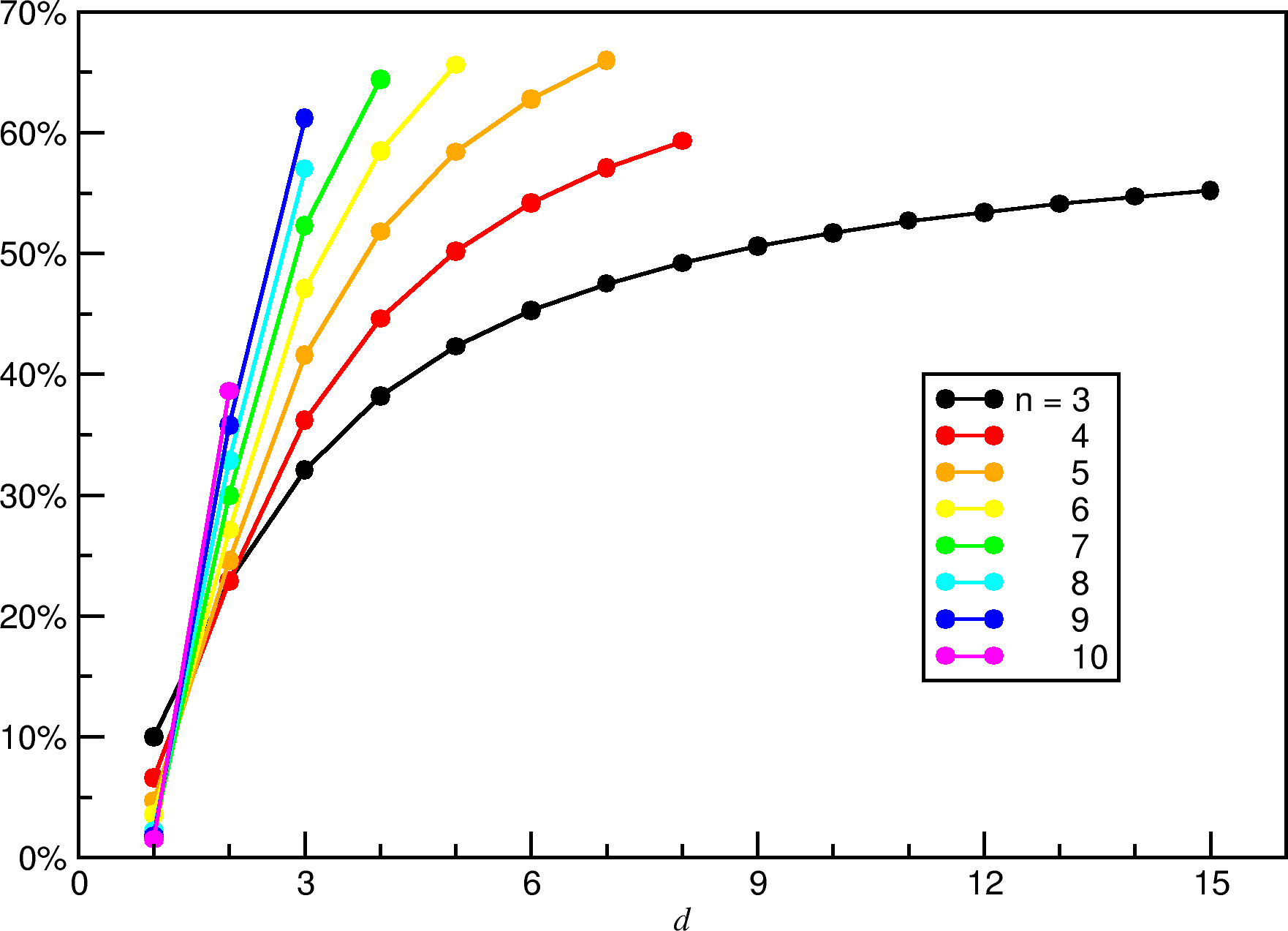}
\caption{Gaphic representation of $\fw_{n,2d}$ from table \ref{tab:Ds}.}\label{fig:lower}
\end{figure}

Some simple facts are collected in the following lemma.

\begin{lem}
\begin{enumerate}[i)]
\item $\fR_{n,2}=1$ and $\fz_{n,2}=1$.
\item $\fw_{n,2} > \fw_{n',2}$ for $1\leq n \leq n'$.
\item $\fR_{n,2d} \leq \fR_{n,2d'}$ and $\fR'_{n,2d} \leq \fR'_{n,2d'}$ for $1\leq d \leq d'$.
\item $\fR_{n,2d} \leq \fR_{n',2d}$ and $\fR'_{n,2d} \leq \fR'_{n',2d}$ for $1\leq n \leq n'$.
\end{enumerate}
\end{lem}

In table \ref{tab:Ds} we observe  that in all calculated cases ($n=3,4,\dots,10$) we have $\fR'_{n,2} = \left(\begin{smallmatrix}n+2\\ 2\end{smallmatrix}\right) - n$. In the next lemma we prove that this holds for all $n\in\nset$.

\begin{lem}\label{lem:n2Rrank}
$\fR'_{n,2} = \left(\begin{smallmatrix}n+2\\ 2\end{smallmatrix}\right) - n = \frac{1}{2}(n^2 + n + 2)$ for $n\in\nset$.
\end{lem}
\begin{proof}
Since $\fq = q_1 + \dots + q_n$, we have $\gamma_{s'} \geq n$ and therefore $\cD_{s'} = \fR'_{n,2} \leq \left(\begin{smallmatrix}n+2\\ 2\end{smallmatrix}\right) - n$. For the converse direction, let $e_i$ be the $i$-th unit vector in $\rset^n$. Take the $\left(\begin{smallmatrix}n+2\\ 2\end{smallmatrix}\right) - n$ points $x_i$ as
\[P=\{0,e_1,\dots,e_n,e_1+e_2,\dots,e_1+e_3,\dots,e_{n-1}+e_n\}\subseteq \{0,1\}^n.\]
Then $\fR'_{n,2} \geq \rank (s_\sA(x))_{x\in P} = \rank M$ with
\[M = \begin{pmatrix}
s_\sA(0)^T\\
s_\sA(e_1)^T - s_\sA(0)^T\\
\vdots\\
s_\sA(e_n)^T - s_\sA(0)^T\\
s_\sA(e_1+e_2)^T-s_\sA(e_1)^T-s_\sA(e_2)^T+s_\sA(0)^T\\
\vdots\\
s_\sA(e_{n-1}+e_n)^T-s_\sA(e_{n-1})^T-s_\sA(e_n)^T+s_\sA(0)^T
\end{pmatrix}.\]
But 
\[s_\sA(0) = \begin{cases}1 & \text{at position}\ 1\\ 0 & \text{elsewhere}\end{cases},\qquad s_\sA(e_i) - s_\sA(0) = \begin{cases}1 & \text{at position}\ x_i\ \text{and}\ x_i^2\\ 0 & \text{elsewere}\end{cases},\]
and
\[s_\sA(e_i+e_j)^T-s_\sA(e_i)^T-s_\sA(e_j)^T+s_\sA(0)^T= \begin{cases}1 & \text{at position}\ x_i x_j\\ 0 & \text{else}\end{cases}.\]
So the rows in $M$ are linearly independent and $\rank M = \left(\begin{smallmatrix}n+2\\ 2\end{smallmatrix}\right) - n$.
\end{proof}

For all pairs $(n,d)$ of numbers occurring in table \ref{tab:Ds} and fig.\ \ref{fig:lower} we gather the inequalities:
\begin{subequations}\label{eq:trends}
\begin{align}
\fw_{n,2d} &< \fw_{n,2d'}&& \text{for all}\ 1\leq d < d',\\
\fw_{n,2d} &< \fw_{n',2d}&& \text{for all}\ 3\leq n < n',\\
\fz_{n,2d} &> \fz_{n,2d'}&& \text{for all}\ 1\leq d < d',\\
\fz_{n,2d} &> \fz_{n',2d}&& \text{for all}\ 3\leq n < n',\\
\fw'_{3,2d} &> \fw'_{3,2d'}&& \text{for all}\ 1\leq d < d'\\
\fw'_{n,2d} &< \fw'_{n,2d'}&& \text{for all}\ 2\leq d < d'\ \text{and}\ n\geq 4,\\
\fw'_{n,2d} &< \fw'_{n',2d}&& \text{for all}\ 3\leq n < n',\\
\fz'_{n,2d} &> \fz'_{n,2d'}&& \text{for all}\ 1\leq d < d',\\
\fz'_{n,2d} &> \fz'_{n',2d}&& \text{for all}\ 3\leq n < n'.
\end{align}
\end{subequations}
We don't know whether or not these inequalities hold in general (see Problem \ref{open:exposedDims} below). Do the limits
\begin{align}
\lim_{n\rightarrow\infty} \fw_{n,2d},\quad \lim_{d\rightarrow\infty} \fw_{n,2d},&\quad
\lim_{n\rightarrow\infty} \fw'_{n,2d},\quad \lim_{d\rightarrow\infty} \fw'_{n,2d},\label{eq:limitsW}\\
\lim_{n\rightarrow\infty} \fz_{n,2d},\quad \lim_{d\rightarrow\infty} \fz_{n,2d},&\quad
\lim_{n\rightarrow\infty} \fz'_{n,2d},\quad \lim_{d\rightarrow\infty} \fz'_{n,2d},\label{eq:limitsZ}
\end{align}
exist and if so, what are these limits (see Problem \ref{open:limits})? From Lemma \ref{lem:n2Rrank} we already obtain
\[\lim_{n\rightarrow\infty} \fw_{n,2} = 1.\]

In the next section we will see how the facial structure and especially table \ref{tab:Ds} affect the Carath\'eodory number $\cat_\sA$.

\section{Carath\'eodory numbers}
\label{sec:cara}

By the Richter Theorem (Theorem \ref{thm:richter}) every $s\in\cS_\sA$ has a $k$-atomic representing measure with $k\leq m$. This justifies the following definitions.

\begin{dfn}
The \emph{Carath\'eodory number} $\cat_\sA(s)$ of a moment sequence $s\in \cS_\sA$ is the smallest $k\in \nset_0$ such that $s$ has a $k$-atomic representing measure:
\[\cat_\sA(s) := \min \{k\in\nset \,|\, s\in\range S_{k,\sA}\} \quad\leq m.\]
The  \emph{Carath\'eodory number} $\cat_\sA$ of the moment cone $\cS_\sA$ is the maximum of numbers $\cat_\sA(s)$ for $s\in \cS_\sA$, or equivalently the smallest number such that every $s\in\cS_\sA$ has an at most $k$-atomic representing measure:
\[\cat_\sA:=\max_{s\in\cS_\sA}\cat_\sA(s)=\min\{k\in\nset\,|\,\cS_\sA\subseteq\range S_{k,\sA}\} \quad\leq m.\]
\end{dfn}

For the univariate polynomial moment problem the following classical result is already contained in \cite{richte57}.

\begin{thm}[{\cite[Satz 11]{richte57}}]
For $\cX = \rset$, $[a,b]$, $[a,b)$, $(a,b]$, or $(a,b)$ and $\sA = \{1,x,\dots,x^d\}$ with $d\in\nset$ and $-\infty\leq a < b\leq -\infty$ we have
\begin{equation}\label{eq:m+12half}
\cat_\sA = \cN_\sA = \left\lceil\frac{d+1}{2}\right\rceil.
\end{equation}
\end{thm}

Here, as usual, $\lceil r\rceil$ denotes the smallest integer which is larger or equal to $r$.

For monomials on $\rset$ with gaps, that is, $\sA=\{ 1,x^{d_2},\dots,x^{d_m}\}$ with $0<d_2<\cdots<d_m$, formula $\cat_\sA = \cN_\sA$ is no longer valid in general. Sufficient conditions for $\cat_\sA = \cN_\sA$ to hold are given in  \cite[Theorem 45]{didio17Cara}. We restate this result without proof. Let
\begin{equation}
\label{eq:oneDimCondA}
\sA = \{1,x^{d_1},\dots,x^{d_m}\}\ \text{with}\ 1\leq d_1 < \dots < d_m = 2d\ \text{on}\ \cX = \rset.
\end{equation}
In \cite[Lem.\ 40]{didio17Cara} it was shown  that
\begin{equation*}
\begin{split}
& \det DS_{k,{\sA}}(c_1,\dots,c_k,x_1,\dots,x_k)=\\
& c_1\cdots c_k \cdot (x_1\cdots x_k)^{2d_1} \prod_{1\leq i < j \leq k} (x_j - x_i)^4 \cdot q_{\sA}(x_1,\dots,x_k),
\end{split}
\end{equation*}
and
\begin{equation*}
\begin{split}
& \det (DS_{k-1,{\sA}},s_{\sA}(x_k))=\\
& c_1\cdots c_{k-1}\cdot (x_1\cdots x_{k-1})^{2d_1} x_k^{d_1} \cdot \prod_{1\leq i < j \leq k-1} (x_j - x_i)^4 \cdot \prod_{i=1}^{k-1} (x_k - x_i)^2 \cdot q_{{\sA},k}(x_1,\dots,x_k),
\end{split}
\end{equation*}
where $q_\sA$ and $q_{\sA,i}$ are polynomials with non-negative coefficients.

\begin{thm}[{\cite[Thm.\ 45]{didio17Cara}}]\label{thm:onedimCara}
Let  $\sA$ be as in (\ref{eq:oneDimCondA}) and $\cZ := \cZ(q_{{\sA}})\subseteq\rset^k$ if $m=2k$ is even and $\cZ := \cZ(q_{\sA,1})\cap \dots\cap \cZ(q_{{\sA},k})\subseteq \rset^k$ if $m=2k-1$ is odd. Suppose
\begin{equation}\label{condition:one-dim-Caratheodory}
(x_1,\dots,x_k)\in\cZ\folgt\exists i\neq j: x_i = x_j.
\end{equation}
Then
\begin{equation}\label{eq:oneDimCat}
\cat_\sA = \left\lceil \frac{m}{2}\right\rceil.
\end{equation}
\end{thm}

An example where (\ref{eq:oneDimCat}) fails is given in \cite[Example 48]{didio17Cara}. 
For one-dimensional monomial systems with gaps, A.\ Wald \cite{wald39}  proved already in 1939 the following.

\begin{thm}[{\cite[Prop.\ 13]{wald39}}]
Let $0\leq d_1 < \dots < d_m$ be integers, $\cX = [0,\infty)$, and $\sA = \{x^{d_1},\dots,x^{d_m}\}$. Then any $s\in\cS_\sA$ has a $k$-atomic representing measure with $k \leq \frac{m+1}{2}$.
\end{thm}

The next result deals with the half-line $\cX = (0,\infty)$.

\begin{thm}[{\cite[Thm.\ 40]{didio18gaussian}}]\label{thm:oneDimGapsOpen}
Suppose  $d_1,\dots,d_m\in\nset_0$, $m\in \nset,$\, $d_1 < \dots < d_m$, and $\sA = \{x^{d_1},\dots,x^{d_m}\}$ on $\cX = (0,\infty)$. Then $\cat_\sA = \left\lceil\frac{m}{2}\right\rceil$.
\end{thm}

For other cases than one-dimensional truncated power moment problems only a few Cara\-th\'eo\-dory numbers are  known, but upper and lower bounds can be given.

First we note that in the general case the bound $\cat_\sA\leq m$ in Theorem \ref{thm:richter} is sharp as we see from the next result. Recall that $m = \dim\cS_\sA$.

\begin{thm}[{\cite[Thm.\ 3.31]{didioDiss}}]\label{thm:CaraCountable}
If $s_\sA(\cX)$ is countable, then
\[\cat_\sA = m.\]
\end{thm}
\begin{proof}
Since $\sA$ is linearly independent, $\cS_\sA$ is full-dimensional and $\inter\cS_\sA \neq\emptyset$. Set $h_{x_1,\dots,x_{m-1}} := \lin \{s_\sA(x_1),\dots,s_\sA(x_{m-1})\}$ with $x_i\in\cX$. Then $h_{x_1,\dots,x_{m-1}}$ is a closed subspace of $\rset^m$ of dimension at most $m-1$. Since $s_\sA(\cX)$ is countable, so is $s_\sA(\cX)^{m-1}$ and $H := \bigcup_{x_1,\dots,x_{m-1}\in\cX} h_{x_1,\dots,x_{m-1}}$ is a countable union of closed subspaces with dimension at most $m-1$. Hence $H$ does not contain inner points. Therefore, $\inter\cS_\sA\setminus H \neq \emptyset$. Any sequences $s\in\inter\cS_\sA\setminus H$ need at least $m$ atoms, since otherwise it would be contained in some hyperplane $h_{x_1,\dots,x_{m-1}}$.
\end{proof}

The truncated moment problem on $\cX = \nset_0$ was studied in \cite{infusino17} and the previous theorem completely solves the Carath\'eodory number problem in this case:  Since $\nset_0$ is countable, so is $s_\sA(\nset_0)$   and therefore $\cat_\sA = m$ by Theorem \ref{thm:CaraCountable}.

A slightly better upper bound than $m$ is given in the following result, see \cite[Thm.\ 12]{didio18gaussian}. It is a version of \cite[Thm.\ 13]{didio17Cara} with weaker conditions, but its proof is verbatim the same.

\begin{thm}\label{thm:m-1Bound}
If the set $\frac{s_\sA}{\|s_\sA\|}(\{x\in\cX \,|\, {s_\sA(y)\neq 0}\})$ has at most $m-1$ path-connected components, then
\[\cat_\sA \leq m-1.\]
\end{thm}

For (homogeneous) polynomials on $\rset^2$ or $\pset^2$ upper bounds for the Carath\'eodory number $\cat_\sA$ have been obtained in \cite{didio17Cara} and \cite{rienerOptima}. They are based on zeros of nonnegative polynomials and use deep results of Petrovski \cite{petrow38} on Hilbert's 6th problem. We summarize the main results in the following theorem.

\begin{thm}\label{thm:goodUpperBounds}
\begin{enumerate}[i)]
\item For $\cX = \rset^2$ we have $\cat_{\sA_{2,2d-1}} \leq \frac{3}{2}d(d-1) + 1$ for $d\in\nset$.
\item For $\cX = \pset^2$ we have $\cat_{\sB_{2,2d}} \leq \frac{3}{2}d(d-1)+2$ for $d\geq 5$.
\end{enumerate}
\end{thm}
\begin{proof}
i) is \cite[Cor.\ 8.4]{rienerOptima} and ii) is \cite[Thm.\ 62]{didio17Cara}.
\end{proof}

A general lower bound for sufficiently differentiable functions $a_i$ was given in \cite[Thm.\ 27]{didio17Cara}. It is  based on Sard's Theorem \cite{sard42}.

\begin{thm}[{\cite[Prop.\ 23 and Thm.\ 27]{didio17Cara}}]\label{thm:NAlowerCat}
Suppose that ${\sA}\subset C^r(\rset^n,\rset)$ with  $r > \cN_{\sA}\cdot (n+1) - m$. Then
\begin{equation}\label{eq:NAlessequalCA}
\left\lceil \frac{m}{n+1} \right\rceil\quad \leq \quad \cN_{\sA}\quad \leq\quad \cat_{\sA}.
\end{equation}
Further, the set of moment sequences $s$ which can be represented by less than $\cN_{\sA}$ atoms has $|{\sA}|$-dimensional Lebesgue measure zero in $\rset^m$.
\end{thm}

Though the previous theorem was stated for  $\cX=\rset^n$,  it remains for differential manifolds by Remark \ref{rem:manifold}.

In the preceding we mainly  reviewed the recent developments on Carath\'eodory numbers from \cite{didio17Cara} and \cite{rienerOptima}. Now we apply the considerations on the facial structure of the moment cone $\cS_\sA$ from the previous section to derive some new results.

\begin{lem}\label{lem:CaraDsBound}
$\cat_\sA(s) \leq \cD_s$ for all $s\in\cS_\sA$.
\end{lem}
\begin{proof}
The assertion follows from Richter's Theorem \ref{thm:richter} applied to $\cX = \cW(s)$.
\end{proof}

\begin{cor}\label{cor:catBoundaryPoly}
Suppose $\sA = \sA_{n,d}$ on $\cX=\rset^n$ or $\sA = \sB_{n,d}$ ($d$ even) on $\cX=\pset^n$. Then, for $s\in\partial^*\cS_\sA := \partial\cS_\sA\cap\cS_\sA$, we have
\[\cat_\sA(s) \leq \begin{pmatrix} n+d\\ n\end{pmatrix} - n-1.\]
\end{cor}
\begin{proof}
Combine Lemma \ref{lem:CaraDsBound} and Proposition \ref{prop:DsBoundBoundary}.
\end{proof}

\begin{thm}\label{thm:projective-n}
For $\sA = \sB_{n,2d}$ on $\cX=\pset^n$ we have
\[\cat_\sA \leq \begin{pmatrix} n+2d\\ n\end{pmatrix} - n.\]
\end{thm}
\begin{proof}
Let $s\in\cS_\sA$. By \cite[Prop.\ 8]{didio17Cara} we can write $s = c\cdot s_\sA(x) + s'$ for some $x\in\pset^n$ and $c>0$ such that\ $s'\in\partial\cS_\sA = \partial^*\cS_\sA$. Then, by Corollary \ref{cor:catBoundaryPoly}, $\cat_\sA(s') \leq \begin{pmatrix} n+2d\\ n\end{pmatrix} - n-1$. Therefore, $\cat_\sA(s) \leq \begin{pmatrix} n+2d\\ n\end{pmatrix} - n$ which implies the assertion.  
\end{proof}

This slightly improves the upper bounds in the projective case. Lower bound improvements are also possible by using table \ref{tab:Ds} and the results obtained at the end of Section \ref{sec:face}.

\begin{thm}\label{thm:zeroLowerBoundCara}
Let $\sA = \sA_{n,2d}$. For  an open subset  $\cX$ of $\rset^n$  or  $\cX = [0,1]^n$, we have
\[\cat_\sA \geq \fR_{n,2d}\ \text{or}\ \fR'_{n,2d},\ \text{respectively}.\]
\end{thm}
\begin{proof}
For $\cX\subseteq\rset^n$ we take $\fp$ from (\ref{eq:p}) and consider the moment problem on $\cX' = \cZ(\fp)$. For $\cX = [0,1]^n$ we take $\fq$. Then Theorem \ref{thm:CaraCountable} shows  that there is a moment sequence $s\in\cS_\sA$ such that\ $\cat_\sA(s) = \dim \conv\ \cone s_\sA(\cX')$ is  $ \fR_{n,2d}$ or $\fR'_{n,2d}$, respectively.
\end{proof}

Let us briefly discuss these results. From Theorem \ref{thm:NAlowerCat} we recall the lower bound 
\begin{equation}\label{eq:CaraNaRation}
\frac{\cat_\sA}{|\sA|} \geq \frac{1}{n+1}
\end{equation}
which decreases  with increasing $n$.  Theorem \ref{thm:zeroLowerBoundCara} yields
\begin{equation}\label{eq:RationrankR}
\frac{\cat_\sA}{|\sA|} \geq \fw_{n,2d}\ \text{and}\ \fw'_{n,2d},\ \text{respectively}.
\end{equation}
As seen from table \ref{tab:Ds}, the numbers $\fw_{n,2d}$ and $\fw'_{n,2d}$ give much better estimates than (\ref{eq:CaraNaRation}), but they have to be calculated for each case. For instance,\ in the case $\sA = \sA_{5,14}$ on $\rset^5$ we have
\[\frac{\cat_\sA}{|\sA|} \geq \frac{1}{6}\approx 0.17\quad \text{from (\ref{eq:CaraNaRation}),}\qquad \text{but}\qquad \frac{\cat_\sA}{|\sA|} \geq \fw_{5,14}\approx 0.66\quad \text{from (\ref{eq:RationrankR})},\]
and in the case $\sA = \sA_{10,4}$ on $[0,1]^{10}$ we even have
\[\frac{\cat_\sA}{|\sA|} \geq \frac{1}{11}\approx 0.09\quad \text{from (\ref{eq:CaraNaRation}),}\qquad \text{but}\qquad \frac{\cat_\sA}{|\sA|} \geq \fw'_{10,4}\approx 0.89\quad \text{from (\ref{eq:RationrankR})}.\]

For $d=1$ (i.e., $\sA = \sA_{n,2}$) on $[0,1]^n$ there is the following explicit result.

\begin{thm}\label{carboundcas}
For $\sA = \sA_{n,2}$ on $\cX=[0,1]^n$ we have
\[\begin{pmatrix} n+2\\ 2\end{pmatrix} - n \quad\leq\quad \cat_\sA \quad\leq\quad \begin{pmatrix} n+2\\ 2\end{pmatrix} - 1\]
and therefore
\[\lim_{n\rightarrow\infty} \frac{\cat_{\sA_{n,2}}}{|\sA_{n,2}|} = 1.\]
\end{thm}
\begin{proof}
The upper bound is Theorem \ref{thm:m-1Bound}, while the lower bound is Theorem \ref{thm:zeroLowerBoundCara} combined with Lemma \ref{lem:n2Rrank}. The limit follows by a straightforward calculation.
\end{proof}

From Theorem \ref{carboundcas} we see that $\left(\begin{smallmatrix} n+2d\\ n\end{smallmatrix}\right) - n$ is a lower bound on the Carath\'eodory number for $\sA = \sA_{n,2}$ on $\cX = [0,1]^n$, but it is also an upper bound for $\sA = \sB_{n,2d}$ on $\cX = \pset^n$  by Theorem \ref{thm:projective-n}. In fact, $\left(\begin{smallmatrix} n+2d\\ n\end{smallmatrix}\right) - n$ is also a lower bound on the  face dimension of the moment cone $\cS_\sA$ for $\sA = \sA_{n,2}$ ($d=1$) on $\cX = [0,1]^n$, but for $\sA = \sA_{n,2d}$ or $\sB_{n,2d}$ on $\cX = \rset^n$ or $\pset^n$ we have an upper bound of the face dimension of $\left(\begin{smallmatrix} n+2d\\ n\end{smallmatrix}\right) - n - 1$ by Corollary \ref{cor:catBoundaryPoly}. Thus, changing the set $\cX$ from $\rset^n$ (or $\pset^n$) to $[0,1]^n$ has drastic effects on  the moment cone and its Carath\'eodory number.

To demonstrate the drastic effect of higher dimensions $n$ we give also the following flat extension example.

\begin{exm}\label{exm:flatextension}
Let $(n,d) = (5,7)$ and $s$ be from table \ref{tab:Ds}, i.e., $s = \sum_{x\in\cZ(\fp)} s_\sA(x)$ is from the end of Section \ref{sec:face}. Then all $11628$ moments of $s$ are collected in a $792\times 792$ Hankel matrix since $\left(\begin{smallmatrix} 7+5\\ 5\end{smallmatrix}\right) = 792$. But from table \ref{tab:Ds} we find that $\cat_\sA(s) = 7678$, i.e., $s$ needs at least $7678$ atoms in a representing measure. So applying flat extension we have to extend the original $792\times 792$ Hankel matrix to an at least $7679\times 7679$ Hankel matrix. We have all moments up to degree $2d = 14$ and must extend them to at least degree $24$ since $\left(\begin{smallmatrix} 12+5\\ 5\end{smallmatrix}\right) = 6188$ but at most degree $26$ since $\left(\begin{smallmatrix} 13+5\\ 5\end{smallmatrix}\right) = 8568$ if all additional $k$ moments are optimally chosen. Then $k$ is bounded by
\[\begin{pmatrix} 24+5\\ 5\end{pmatrix} - \begin{pmatrix} 14+5\\ 5\end{pmatrix} = 107\,127\quad\leq\quad k\quad\leq\quad \begin{pmatrix} 26+5\\ 5\end{pmatrix} - \begin{pmatrix} 14+5\\ 5\end{pmatrix} = 158\,283.\]
\end{exm}

The previous example shows that the application of flat extension to larger systems might not be possible.

There are several reasons which damp the hope of finding upper bounds that are significantly lower than those given in Theorem \ref{thm:goodUpperBounds}. First, the proofs in \cite{didio17Cara} and \cite{rienerOptima} are tight and based on Petrovski's deep result \cite{petrow38}; it seems hardly possible to improve the corresponding bounds in this manner since the number of isolated zeros exceed the number $m$ of monomials as seen from table \ref{tab:Ds}. Secondly, Theorems \ref{thm:zeroLowerBoundCara} and \ref{carboundcas} combined with the lower bounds in table \ref{tab:Ds} indicate that strong improvement cannot be expected, see also the growth of the lower bounds $\fw_{n,2d}$ in fig.\ \ref{fig:lower}. This indicates that further investigations of the inequalities in (\ref{eq:trends}) and the possible limits in (\ref{eq:limitsW}) are important.

In table \ref{tab:Ds}  only the simple polynomials $\fp$ and $\fq$ with large but finite numbers of zeros have been used. It is natural to ask whether or not the lower bounds of the Carath\'eodory number $\cat_\sA$ in table \ref{tab:Ds} can be (significantly) improved by using other non-negative polynomials with finitely many zeros (see Problem \ref{open:otherPolynomials})?

Another variant of the Carath\'eodory number problem  is to allow  \emph{signed} measures and to study \emph{signed} Carath\'eodory numbers $\cat_{\sA,\pm}$. By Proposition \ref{prop:signedMeasures},  every vector $s\in\rset^m$ has a representing $k$-atomic measure with $k\leq m$. This leads to the following definition.

\begin{dfn}\label{dfn:signedCat}
The \emph{signed Carath\'eodory number $\cat_{\sA,\pm}(s)$} of $s\in\rset^m$ is 
\[\cat_{\sA,\pm}(s) := \min \{k\in\nset \,|\, s \in S_{k,\sA}(\rset^k\times\cX^k)\}\]
and the \emph{signed Carath\'eodory number $\cat_{\sA,\pm}$} is
\[\cat_{\sA,\pm} := \min \{k\in\nset \,|\, S_{k,\sA}(\rset^k\times\cX^k) = \rset^m\} \equiv \max_{s\in\rset^m} \cat_{\sA,\pm}(s).\]
\end{dfn}

For the signed Carath\'eodory number $\cat_{\sA,\pm}$ we have the following result.

\begin{thm}[{\cite[Thm.\ 25]{didio17Cara}}]\label{thm:signedupperBound}
Suppose $\cX$ is an open subset $\rset^n$ and $a_i\in C^1(\cX,\rset)$ for all $i$. Then
\[\cat_{\sA,\pm} \leq 2\cN_\sA.\]
\end{thm}

Comparing (\ref{eq:CaraNaRation}) and (\ref{eq:RationrankR}) we see that  $\fw_{n,2d}$  gives a much better lower bound than $\cN_\sA$. For the signed Carath\'eodory number $\cat_{\sA,\pm}$ we have the upper bound $2\cN_\sA$ by Theorem \ref{thm:signedupperBound}.  Combining Theorem \ref{thm:signedupperBound} with Theorem \ref{thm:m-1Bound} we get
\begin{equation}\label{eq:uppersignedCombined}
\cat_{\sA,\pm} \leq \min\{2\cN_\sA,m-1\}.
\end{equation}
In \cite[Sec.\ 7]{didio17Cara} we used the apolar scalar product to relate the signed Carath\'eodory number $\cat_{\sA,\pm}$ to the real Waring rank. Note that (\ref{eq:uppersignedCombined}) gives a sharp bound, since G.\ Blekherman \cite{blekhe15a} showed that there is a $f\in\rset[x,y]_{d}$ which can be written as
\[f(x,y) = \sum_{i=1}^k c_i(a_i x + b_i y)^d\]
with $k=d$, but not with $k<d$.

\section{Internal structure of $\cS_\sA$}
\label{sec:inner}

In Definition \ref{dfn:momentCurveMap} we already defined the moment map $S_{k,\sA}$. For $\cX\subseteq\rset^n$ open and $a_i\in C^1(\cX,\rset)$ the moment map  is a $C^1$-mapping and we can investigate the atomic measures $\mu = (C,X) = \sum_{i=1}^k c_i \delta_{x_i}$. The following important definition is taken from \cite[Def.\ 26]{didio17Cara}.

\begin{dfn}\label{dfn:regular}
An atomic measure $\mu_{(C,X)} = \sum_{i=1}^k c_i \delta_{x_i}$ is called
\begin{itemize}
\item \emph{regular} if the matrix $DS_{k,\sA}(C,X)$ is regular (i.e., it has full rank),
\item \emph{singular} if the matrix $DS_{k,\sA}(C,X)$ is singular (i.e., does not have full rank).
\end{itemize}
A moment sequence $s\in\cS_\sA$ is called \emph{regular} iff $S_{k,\sA}^{-1}(s)$ is empty or consists only of regular measures. Otherwise, $s$ is called \emph{singular}.
\end{dfn}

Being singular or regular might depend on the number $k$ according to the preceding definition. Obviously, if $s$ is singular for $k$, so  is for all $k'\geq k$. We dont know whether or not  is it possible that $s$ is regular for  $k$ and  singular for some $k'>k$ (see Problem \ref{open:regular} below)?

\begin{exm}\label{exm:completeoneDim}
In the one-dimensional case $\sA = \{1,x,\dots,x^d\}$ on $\cX=\rset$ it follows from \cite[Lem.\ 35]{didio17Cara} that
\begin{equation}\label{eq:interregular}
s\in\inter\cS_\sA \quad\Leftrightarrow\quad s\ \text{is regular}.
\end{equation}
\end{exm}

In this section we want to investigate this relation (\ref{eq:interregular}).

For the one-dimensional polynomial case $\sA = \{1, x^{d_2},\dots,x^{d_m}\}$ ($d_m = 2m$) with gaps we observed in \cite[Thm.\ 45]{didio17Cara} that when $DS_{k,\sA}(\mu)$ is singular then $k < \lceil\frac{m}{2}\rceil$ if the condition in \cite[eq.\ (32)]{didio17Cara} is fulfilled. But a singular moment sequence might be an inner point of $\cS_\sA$, as the following example shows.

\begin{exm}\label{exm:InterSingular}
Let $\sA = \{1,x^2,x^3,x^5,x^6\}$ on $\cX=\rset$. In {\cite[Exm.\ 46]{didio17Cara}} we showed with Theorem \ref{thm:onedimCara} that $\cat_\sA = \cN_\sA = 3$. We will prove here that (\ref{eq:interregular}) does not hold, so there exists an inner point of $\cS_\sA$ which  is singular.

Let $x_1= 1$,  $x_2=2$ and $s = c_1 s_\sA(x_1) + c_2 s_\sA(x_2)$ for some $c_1,c_2>0$. Then $s$ is singular, since $\rank DS_{2,\sA}((c_1,c_2),(x_1,x_2)) = 4 < 5 = |\sA|$. But $s$ is not a boundary moment sequence. If $s$ would be a boundary sequence, then there exists  a $p\in\pos(\sA)$ with $x_1,x_2\in\cZ(p)$. From
\[\ker DS_{2,\sA}(\one,(x_1,x_2))^T = \ker \begin{pmatrix}
1 & 1 & 1 & 1 & 1\\
0 & 2 & 3 & 5 & 6\\
1 & 4 & 8 & 32 & 64\\
0 & 4 & 12 & 80 & 192
\end{pmatrix} = v\cdot\rset\]
with $v= (52,-231,225,-63,17)^T$, we find that
\[p(x) = \langle v,s_\sA(x)\rangle = (x-1)^2 (x-2)^2 (17x^2 + 39x + 13)\]
is the only possible polynomial in $\cA$ with $p(x_i) = p'(x_i) = 0$. But it is indefinite  since $p(-1)=-324$ and $p(0) = 52$, i.e., $s$ is no boundary point but an inner point.
\end{exm}

Thus, Example \ref{exm:InterSingular} is an example where $\cat_\sA = \cN_\sA$ holds but (\ref{eq:interregular}) does not. Next we show that it is also possible that (\ref{eq:interregular}) holds but $\cat_\sA \neq \cN_\sA$.

\begin{exm}\label{exm:BoundarySingular}
Let $\sA = \{1,x,x^2,x^6\}$ on  $\cX=\rset$. Then we have seen in \cite[Exm.\ 48]{didio17Cara} that\, $\cat_\sA = 3 > \cN_\sA = 2$. By a straightforward calculation we verify that
\[\det DS_{2,\sA}(\one,(x,y)) = 2 (x - y)^4 q_\sA(x,y) \quad\text{with}\quad q_\sA(x,y)= (x + y) (2 x^2 + x y + 2 y^2)\]
and $\cZ(q_\sA) = \{(a,-a) \,|\, a\in\rset\}$. Hence $c_1 s_\sA(x) + c_2 s_\sA(y)$ is regular iff $x\neq y$ or $-y$. 

Let $s$ be singular. Then $s = c_1 s_\sA(a) + c_2 s_\sA(-a)$ and $s\in\partial^*\cS_\sA$, since
\[p_a(x) := x^6  - 3a^4 x^2 + 2a^6 = (x-a)^2(x+a)^2(2a^2+x^2) \in\pos(\sA),\]
that is, (\ref{eq:interregular}) holds.
\end{exm}

At last, we  give an example where neither $\cat_\sA = \cN_\sA$ nor (\ref{eq:interregular}) hold.

\begin{exm}\label{exm:InterSingularMulti}
Let
\begin{align*}
\sA = \sB_{2,6} = \{
& x^6, x^5 y, x^4 y^2, x^3 y^3, x^2 y^4, x y^5, y^6, x^5 z, x^4 y z, x^3 y^2 z, x^2 y^3 z,\\
& x y^4 z, y^5 z, x^4 z^2, x^3 y z^2, x^2 y^2 z^2, x y^3 z^2, y^4 z^2, x^3 z^3, x^2 y z^3,\\
& x y^2 z^3, y^3 z^3, x^2 z^4, x y z^4, y^2 z^4, x z^5, y z^5, z^6\}
\end{align*}
on $\cX=\pset^2$. $\cN_\sA = 10$ by Theorem \ref{thm:NApolynomial}, but it is known that $\cat_\sA = 11$ \cite{kunertPhD14}. Let
\begin{align*}
X =  \{ \xi_1,\dots,\xi_{10}\}
= \{&(1, 0, 0),(0, 1, 0),(0, 0, 1),(1, 1, 0),(1, 0, 1),\\
&(0, 1, 1),(1, -1, 0),(1, 0, -1),(1, 1, 1),(1, -1, 1)\}
\end{align*}
be $10$ projective points. Set $s: = \sum_{i=1}^{10} c_i s_\sA(\xi_i)$ for some numbers $c_1,\dots,c_{10} >0$. Since $\rank DS_{10,\sA}(\one,X)=27 < 28 = |\sA|$, the moment sequence $s$ is singular. 

But $s$ is not a boundary moment sequence. Assume the contrary. Then there exists a $p\in\pos(\sA)$ with $X\subseteq\cZ(p)$. From $\ker DS_{10,\sA}(\one,X)^T = v\cdot\rset$ with
\[v = (0, 0, 0, 0, 0, 0, 0, 0, 0, 1, 0, -1, 0, 0, -1, 0, 2, 1, 0, 0, -2, -2, 0, 1, 1, 0, 0, 0)^T\]
we find that
\[p(x) = \langle v,s_\sA(x,y,z)\rangle = y (x + y) (x - z) (y - z) z (x - y + z)\]
or a multiple of it. Since $p(1,2,3)=72$ and $p(6,2,3)=-1008$,  $p$ is indefinite, a contradiction. This proves that $s$ is not a boundary point.

Therefore, $s$ is a singular  inner point of the moment cone. Thus, $\sB_{2,6}$ on $\cX=\pset^2$ is an example such that 
$\cat_\sA\neq \cN_\sA$ and relation (\ref{eq:interregular}) does not hold.
\end{exm}

We summarize the results of all four examples in table \ref{tab:possibleCombinations}.
\begin{table}[!ht]\center
\caption{All possible combinations of ``$\cat_\sA=\cN_\sA$'' and ``$s$ regular $\Leftrightarrow$ $s\in\inter\cS_\sA$'' (eq.\ \ref{eq:interregular}).}\label{tab:possibleCombinations}
\begin{tabular}{rc|c|c}
\hline\hline
& & \multicolumn{2}{c}{\multirow{2}{*}{``$s$ regular $\Leftrightarrow$ $s\in\inter\cS_\sA$''}}\\ 
&\\
&       & TRUE & FALSE \\ \hline
\multirow{4}{*}{``$\cat_\sA=\cN_\sA$''} & \multirow{2}{*}{TRUE} & $\sA=\{1,x,\dots,x^d\}$ & $\sA=\{1,x^2,x^3,x^5,x^6\}$ \\
&       & (Example \ref{exm:completeoneDim}) & (Example \ref{exm:InterSingular})\\ \cline{2-4}
& \multirow{2}{*}{FALSE} & $\sA=\{1,x,x^2,x^6\}$ & $\sA = \sB_{2,6}$\\
& & (Example \ref{exm:BoundarySingular})    & (Example \ref{exm:InterSingularMulti})\\ \hline\hline
\end{tabular}
\end{table}
The table \ref{tab:possibleCombinations} shows that ``$\cat_\sA = \cN_\sA$'' (true T or false F) and (\ref{eq:interregular}) (T or F) are independent properties. Example \ref{exm:InterSingularMulti} indicates that (F,F) is probably the largest and most common class. Our only examples for (T,T) are $\sA=\{1,x,\dots,x^d\}$ on $\rset$ (Example \ref{exm:completeoneDim}) and $\sA = \{x^{d_1},\dots,x^{d_m}\}$ on $(0,\infty)$ (Theorem \ref{thm:oneDimGapsOpen}). Several natural questions arise:

\textit{Are Example \ref{exm:completeoneDim} and Theorem \ref{thm:oneDimGapsOpen} the only examples in (T,T), that is, for which $\cat_\sA = \cN_\sA$ and (\ref{eq:interregular}) holds (Problem \ref{open:internal})? If there are others, can they be completely described? Besides the two independent properties ``$\cat_\sA = \cN_\sA$''  and (\ref{eq:interregular}), are there other basic properties to distinguish moment problems?}

\section{Applications to SOS and Tensor Decompositions}
\label{sec:appl}

In this section and the next we  discuss some applications of the previous results. 
Let us begin with \emph{sums of squares}. The following definition is an adaption of Definition \ref{dfn:momentCurveMap} to the study of sums of squares. For a set $\sA = \{a_1,\dots,a_m\}$ of $m$ elements (e.g.\ measurable function) we set $\sA^2 := \{a_i a_j \;|\; i,j=1,\dots,m\}$.

\begin{dfn}
Let $\sA=\{a_1,\dots,a_m\}$. The \emph{square curve $p_\sA$} is defined by
\[p_\sA:\rset^m\rightarrow\lin\sA^2,\ y=(y_1,\dots,y_m)\mapsto p_\sA(y) := (y_1 a_1 + \dots + y_m a_m)^2\]
and the \emph{square map $P_{k,\sA}$} is
\[P_{k,\sA}:\rset^{k\cdot m}\rightarrow\lin\sA^2, Y=(Y_1,\dots,Y_m)\mapsto P_{k,\sA}(Y) := \sum_{i=1}^k p_\sA(Y_i)\]
with $Y_i=(y_{i,1},\dots,y_{i,m})$. $\Sigma\sA^2$ denotes the sum of squares in $\lin\sA^2$.
\end{dfn}

The following lemma collects straightforward results adapted from results in the previous sections.

\begin{lem}\label{lem:fulldimCones}
Let $\sA = \{a_1,\dots,a_m\}$.
\begin{enumerate}[i)]
\item \label{item:bounds} $m \leq \dim\lin\sA^2 \leq m^2$. If $a_i a_j = a_j a_i$, then $\dim\lin\sA^2 \leq \frac{m(m+1)}{2}$.
\item For all $k\geq \dim\lin\sA^2$ we have
\begin{equation}\label{eq:pythagorasNo}
\Sigma\sA^2 = \range P_{k,\sA^2} = P_{k,\sA^2}(\rset^{m\cdot k})
\end{equation}

\item \label{item:sosfulldim} $\Sigma \sA^2$ is a closed full-dimensional cone in $\lin\sA^2$.
\item $\pos(\sA^2)$ is a closed full-dimensional cone in $\lin\sA^2$.
\end{enumerate}
\end{lem}

\begin{dfn}
The \emph{Pythagoras number $\pyt_{\sA^2}$ of $\sA^2$} is
\[\pyt_{\sA^2} := \min \{k\in\nset \;|\; \text{(\ref{eq:pythagorasNo}) holds}\}.\]
\end{dfn}

With the next proposition we illustrate the use of the square map by giving a simple proof of the well-known fact that the Pythagoras number of $\rset[x_1,\dots,x_n]$ is $\infty$ \cite[Sec.\ 8.1]{prestelPosPoly}. This proof follows some arguments from \cite[Prop.\ 23 and Thm.\ 27]{didio17Cara}.

\begin{prop}\label{prop:pythagoras}
\begin{enumerate}[i)]
\item $\left\lceil\frac{|\dim\lin\sA^2|}{|\sA|}\right\rceil \leq \pyt_{\sA^2}$.
\item  For $n\geq 2$,  the Pythagoras number of $\rset[x_1,\dots,x_n]$  is $\infty$.
\end{enumerate}
\end{prop}
\begin{proof}
i): The map $P_{k,\sA}:\rset^{k\cdot m}\rightarrow\lin\sA^2$ is a $C^\infty$-map. Since $\Sigma\sA^2$ is full-dimensional by Lemma \ref{lem:fulldimCones}\,\ref{item:sosfulldim}), it has a non-empty interior. By Sard's Theorem \cite{sard42} the set of regular values in $\Sigma\sA^2$ is dense. Hence $k\geq \left\lceil\frac{|\dim\lin\sA^2|}{|\sA|}\right\rceil$.

ii): Set $\sA = \sA_{n,d}$. Then $\lin\sA^2 = \lin\sA_{n,2d}$. So for $\Sigma_{n,2d}$ with $n\geq 2$ there is a sum of squares with at least $\left\lceil\frac{|\sA_{n,2d}|}{|\sA_{n,d}|}\right\rceil$ many squares and the required number of squares does not decrease with increasing $d$ since higher terms have non-negative coefficients. Therefore, in $\rset[x_1,\dots,x_n]$ we have
\[\left\lceil\frac{|\sA_{n,2d}|}{|\sA_{n,d}|}\right\rceil = \left\lceil \begin{pmatrix}n+2d\\ d\end{pmatrix}\cdot \begin{pmatrix}n+d\\ n\end{pmatrix}^{-1}\right\rceil \xrightarrow{n\rightarrow\infty}\infty.\qedhere\]
\end{proof}

Using algebraic versions of Sard's Theorem (see e.g.\ \cite[Sec.\ 9.6]{bochnak98}) the preceding proof carries over to $R[x_1,\dots,x_n]$ for $R$ a real closed field. For $n=1$ (i.e., $\sA = \sA_{1,d} = \{1,x,\dots,x^d\}$ for all $d\in\nset$)  the minimal number of squares in $\rset[x]$ is
\[\left\lceil \frac{|\sA_{1,2d}|}{|\sA_{1,d}|}\right\rceil = \left\lceil\frac{2d+1}{d+1}\right\rceil = 2 \qquad\forall d\in\nset.\]
This is also the maximal number which is needed \cite[Prop.\ 1.2.1]{marshallPosPoly}. 

The following example shows that for univariate polynomials  with gaps the Pythagoras number can be arbitrary large.

\begin{exm}\label{exm:pythlargegaps}
Let $\sA = \{1,x,x^3,x^7\}$, then $\sA^2 = \{1,x,x^2,x^3,x^4,x^6,x^7,x^8,x^{10},x^{14}\}$, i.e., $|\sA| = 4$ and $|\sA^2| = 10$. The upper bound in Lemma \ref{lem:fulldimCones}\,\ref{item:bounds}) is attained since $10 = \frac{4(4+1)}{2}$. Hence, by Proposition \ref{prop:pythagoras}, 
\[\pyt_{\sA^2} \geq \left\lceil\frac{10}{4}\right\rceil = 3.\]
There  is a sum of squares  $(y_1 + y_2\cdot x + y_2\cdot x^3 + y_4\cdot x^7)^2$ which cannot be written as a  sum of less than $3$ squares. 

In general, by choosing appropriate numbers $d_1 < d_2 < \dots < d_m$, $d_i\in\nset_0$, we have $|\sA^2| = \frac{m(m+1)}{2}$, so that $\frac{m+1}{2}\leq \pyt_{\sA^2}$ and hence $\pyt_{\sA^2}\rightarrow\infty$ as $m\rightarrow\infty$.
\end{exm}

Of course, the preceding example is equivalent to $\tilde{\sA} = \{1,x,y,z\}$ when we set $y = x^3$ and $z = x^7$ and no cancellations in $(\tilde {\sA})^2$ appear. Thus, studying  univariate cases with gaps might give new insight into  multivariate cases. Which $\pyt_\sA=k\in\nset$ can be realized is an open problem (Problem \ref{open:pythagoras}).

For univariate polynomials with gaps nonnegative polynomials are not necessarily sum of squares, as shown by the next example.

\begin{exm}\label{exm:notSOS}
Let $\sA = \{1,x^2,x^3\}$. Then $\sA^2 = \{1,x^2,x^3,x^4,x^5,x^6\}$ and
\[p(x) := x^6 - x^4 + 10\in\pos(\sA^2)\setminus\Sigma\sA^2.\]
That $p$ is not a sum of squares follows immediately from
\[(a + bx^2 + cx^3)^2 = a^2 + 2abx^2 + 2acx^3 + b^2x^4 + 2bcx^5 + c^2x^6\]
since the coefficient $b^2$ of $x^4$ is non-negative.
\end{exm}

So, again, univariate polynomial systems with gaps bear properties of the multivariate polynomials.\medskip

Our second application concerns the \emph{tensor decomposition}, see e.g.\ \cite{bernar13}. A tensor $T$ is a multilinear map
\[T:\rset^{n_1}\times\dots\times\rset^{n_d}\rightarrow\rset \quad\text{with}\quad T=(t_{i_1,\dots,i_d})_{\substack{ i_1=1,\dots,n_1\\\dots \\ i_d=1,\dots,n_d}} \in\rset^{n_1\times\dots\times n_d}.\]
The simplest tensor is the rank $1$ tensor
\[x_1\otimes \dots\otimes x_d := (x_{1,i_1}\cdots x_{d,i_d})_{\substack{ i_1=1,\dots,n_1\\\dots \\ i_d=1,\dots,n_d}}\]
with $x_i =(x_{i,1},\dots,x_{i,n_i})\in\rset^{n_i}$. In order to bring tensors into the framework of moment problems  we set $\sA = \{x_{1,i_1}\cdots x_{d,i_d} \;|\; i_j=1,\dots,n_j,\ j=1,\dots,d\}$. Then the tensor decomposition
\begin{equation}\label{eq:tensordecomp}
T = \sum_{j=1}^k v_{j,1}\otimes\dots\otimes v_{j,d} \quad\text{with}\quad v_{j,i}\in\rset^{n_i}
\end{equation}
is nothing but determining  a signed representing $k'$-atomic measure $k'\leq k$. From Theorem \ref{thm:NAlowerCat} we get the lower bound on the number $R$ of rank $1$ tensors
\[\left\lceil\frac{n_1\cdots n_d}{n_1 + \dots + n_d - d+1}\right\rceil \leq R.\]
Thus, finding a signed atomic representing measure (or an approximation) for $T$ is equivalent to finding a tensor decomposition (\ref{eq:tensordecomp}) (or  an approximation).

\section{Maximal masses and conic optimization}
\label{sec:maxmass}

\begin{dfn}
For $s\in \cS_\sA$ the \emph{maximal mass function} $\rho_s$ is 
\[\rho_s(x) := \sup \{\mu (A_x) \,|\, \mu\in \fM_\sA(s)\}, \quad x\in \cX,\]
with $A_x := s_{\sA}^{-1}(s_\sA(x)) = \{y\in\cX \,|\, s_\sA(x) = s_\sA(y)\}$.
\end{dfn}

We have $\cW(s)=\{ x\in \cX \,|\, \rho_s(x)>0\}$ from Lemma \ref{lem:setofatoms}. Another important quantity is defined by
\begin{equation}\label{eq:rhoLm2}
\kappa_s(x) =\inf\{\ L_s(p) ~|~ p\in {\Pos}(\sA),\ p(x)=1\}
\end{equation}
where\, $\frac{c}{0}:=+\infty$ for $c\geq 0.$

The nonnegative number $\kappa_s(x)$ in (\ref{eq:rhoLm2}) is defined by a conic optimization problem: It is the infimum of the Riesz functional $L_s$ over  the cone ${\Pos}({\sA},\cK)$ under the constraint $p(x)=1$. The following results is contained in \cite[Prop.\ 8]{didio17Cara}.

\begin{prop}\label{prop:kappaGleichRho}
Suppose the moment cone $\cS_\sA$ is pointed (that is, line-free) and $x\in\cX$. Then
\begin{equation}\label{eq:rhoequpiL}
\rho_s(x) = \sup \{ c\in \rset_+ ~|~ s- c\cdot s_\sA(x) \in \cS_\sA \} \leq \kappa_s(x) < \infty.
\end{equation}
If $\cS_\sA$ is also closed, then $\kappa_s(x) = \rho_s(x)$, the supremum in (\ref{eq:rhoequpiL}) is a maximum, and
\[s - \rho_s(x)s_\sA(x)\in\partial\cS_\sA.\]
\end{prop}
\begin{proof}
The first equality in (\ref{eq:rhoequpiL}) is clear from Lemma \ref{lem:setofatoms}. Let $p\in {\Pos}(\sA)$ and  $p(x)=1$. Then 
\[L_s(p) = \int_\cX p(y)~\diff\mu(y) \geq p(x)\mu(A_x) = \mu(A_x).\]
Taking the infimum over $p$ and the supremum over $\mu\in\fM_\sA(s)$ we get $\rho_s(x) \leq \kappa_s(x)$.

Suppose $\cS_\sA$ is closed and pointed. From \cite[Lem.\ A.40]{schmudMomentBook} it follows that the supremum in (\ref{eq:rhoequpiL}) is attained,  $s - \rho_s(x) s_\sA(x)\in\partial\cS_\sA$ and $\kappa_s(x) = \rho_s(x)$.
\end{proof}

If $\sA\subseteq C(\cX,\rset)$, there is an $e\in\lin\sA$ such that $e>0$, and $\cX$ is compact, then the moment cone $\cS_\sA$ is pointed  and closed. Thus, if $\sA$ consists  of continuous functions on a compact space, then the maximal mass function $\rho_s(x)$ can be computed by the conic optimization problem (\ref{eq:rhoLm2}).

Now we turn to the problem when the masses of an atomic representing measure are maximal.

\begin{dfn}
Let $s\in\cS_\sA$ and let $\mu = \sum_{j=1}^k c_j\delta_{x_j}$, $c_j>0$, be a $k$-atomic representing measure of $s$. We say that $\mu$ has \emph{maximal mass} at $x_i$ if $c_i = \rho_s(x_i)$ and that $\mu$ is a \emph{maximal mass measure} for $s$ if $c_j=\rho_s(x_j)$ for all $j=1,\dots,k$.
\end{dfn}

Let $\mu=\sum_{j=1}^k c_j\delta_{x_j}$, $c_j>0$, be a  representing measure of $s$. Fix $i\in \{1,\dots,k\}$. Suppose there exists a function $f\in \Pos (\sA)$ such that $p(x_i)=1$ and $p(x_j)=0$ for all $j\neq i$. Then, for any $\nu\in \fM_\sA(s)$, we have
\begin{equation}\label{eq:maxmassequ}
\nu(A_{x_i})\leq \int p(x)~ d\nu=L_s(p)=\int p(x) ~d\mu= c_i. 
\end{equation}
Therefore, $c_i=\rho_s(x_i)$ and $\mu$ has maximal mass at $x_i$.

\begin{dfn}\label{dfn:pspproperty0}
We say that points $x_1,\dots,x_k\in \cX$ satisfy the \emph{positive separation property} $(PSP)_{\sA}$ if there exist functions $p_1,\dots,p_k\in\Pos(\sA)$ such that 
\begin{equation}
p_i(x_j) = \delta_{i,j}\quad \text{for}\quad i,j=1,\dots,k.
\end{equation}
\end{dfn}

By the reasoning preceding Definition \ref{dfn:pspproperty0} it follows that if $x_1,\dots,x_k\in\cX$ obey  $(PSP)_{\sA}$, then $\mu = \sum_{j=1}^k c_j \delta_{x_j}\in \fM_\sA(s)$ is a maximal mass measure.

If  $\cS_\sA$ is closed and pointed and $s$ is an inner point of $\cS_\sA$, then the converse of the preceding statement is true, that is, if $\mu = \sum_{j=1}^k c_j\delta_{x_j}\in \fM_\sA(s)$ is a maximal mass measure, then the points $x_1,\dots,x_k$ satisfy $(PSP)_\sA$. Indeed, the assumption imply that $\rho_s(x_j) = \kappa_s(x_j)$ and the infimum (\ref{eq:rhoLm2}) for $x=x_j$ is attained at some function $p_j$. One easily verifies that (\ref{eq:maxmassequ}) holds for the functions $p_1,\dots,p_k$. Thus, $x_1,\dots,x_k$ satisfy $(PSP)_\sA$. More details and examples on this matter can be found in \cite[Sec.\ 18.4]{schmudMomentBook}.

The remaining part of this section is devoted to  the question when the infimum in (\ref{eq:rhoLm2}) is a minimum. As noted by \cite[Prop.\ 18.28]{schmudMomentBook}, this holds if $s$ is an interior point of the moment cone. 
The following example shows that for boundary points  the infimum in  (\ref{eq:rhoLm2}) is not necessarily attained.

\begin{exm}\label{exm:kappaNotAttained}
Let $\cX = \rset$, $\alpha\in [1,\infty)$, and $\sA = \{1,x,f_\alpha(x)\}$, where 
\[f_\alpha(x) := \begin{cases} 0 & \text{for}\ x\leq 0\\ x^\alpha & \text{for}\ x>0 \end{cases},\]
comsider the moment sequence $s := (2,-2,0) = s_\sA(0) + s_\sA(-2)$. Then  we have $\rho_s(-2) = \kappa_s(-2) = 1.$

First suppose $\alpha=1$. Set $p(x) = -x/2 + f_\alpha(x)$. Then $p\in\pos(\sA)$, $p(-2)=1$ and $L_s(p)=1$, that is, the infimum in  (\ref{eq:rhoLm2}) at $x=-2$ is attained for $p$.

Now suppose $\alpha > 1$. We  show that the infimum  (\ref{eq:rhoLm2}) for $x=-2$ is not attained.
Assume  the contrary.  Then there exists  $p(x)=a+bx +cf_\alpha(x) \in\pos(\sA)$ such that $L_s(p)=1$, so $2a-2b=1$, and $p(-2) =1$, so $a-2b=1$. Hence $a=0$ and $b=-\frac{1}{2}$. We consider the function $p(x)=-x/2 +cf_\alpha(x) \in\pos(\sA)$ on $(0,\varepsilon)$ for small $\varepsilon>0$ and conclude that $c=0$. Thus,  $p(x)=-x/2  \in\pos(\sA)$, a contradiction.

Since $L_s(f_\alpha)=0$ and $f_\alpha\in \pos(\sA)$, $s$ is  a boundary point of the moment cone.

We  illustrate the  fact that the infimum  (\ref{eq:rhoLm2}) in the case $\alpha >1$ is not attained from a slightly different view point. There exists  a sequence $(p_n)_{n\in\nset}$ of  functions $p_n(x) = a_n + b_n x + c_n f_\alpha(x)$ in $\pos(\sA)$ such that $p_n(-2)=1  $ and $\lim_{n\to \infty} L_s(p_n)=\kappa_s(-2)=1$. From $p_n(-2)=a_n-2b_n$ and $L_s(p_n)=2a_n- 2b_n$ we conclude that $\lim_{n\to \infty} a_n=0$ and $\lim_{n\to \infty} b_n=\frac{1}{2}(a_n-1)=-\frac{1}{2}$. We claim that $\lim_{n\to \infty} c_n=\infty$. Indeed, otherwise there is a subsequence $(c_{n_k}) $ of $(c_n)$ which has a finite limit $c$. Then
\[\lim_{k\to \infty} p_{n_k}(x) = \lim_{k\to \infty} a_{n_k} +b_{n_k} x +c_{n_k} f_{\alpha}(x) = -x/2 +cx^{\alpha}\geq 0\ \text{for small}\ x>0,\]
which is a contradiction. Hence $\lim_n c_n = \infty$, so the sequence $(p_n)$ does not converge.
\end{exm}

In Example \ref{exm:kappaNotAttained} we have seen two boundary points of the moment cone, one for which the infimum (\ref{eq:rhoLm2}) is a minimum, for the other  it is not. The difference  lies in the geometry of moment cone at the boundary point. While  $s_\sA(0)$ is an ``edge'' of the moment cone, the moment cone is ``round'' in a neighborhood of $s$ for $\alpha>1$. This shows that whether or not the optimization problem (\ref{eq:rhoLm2}) has a solution depends on the geometry of the boundary of the moment cone.

\begin{exm}
Let $\cX = \rset$,  $\sA = \{1,x,x^2(x+2)^2\}$, and $\sB = \{1,x,f(x)\}$, where
\[f(x) := \begin{cases} x^2(x+2)^2 & \text{for}\ x>0\ \text{or}\ x<-2\\ 0 & \text{for}\ x\in (-2,0)\end{cases}.\]
Then  $\cS_\sB = \cS_\sA$ and both cones have the same local behavior at $s_\sA(0) = s_\sB(0) = (1,0,0)$. Setting
 $$s := (2,-2,0) = s_\sA(0) + s_\sA(-2)= s_\sB(0) + s_\sB(-2),$$ it follows by a similar reasoning as in Example \ref{exm:kappaNotAttained} that for $\sA$ and  $\sB$  the infimum (\ref{eq:rhoLm2}) at $x=-2$  is not attained.
\end{exm}

The next result  characterizes the case when the infimum (\ref{eq:rhoLm2}) is a minimum. In particular,  it implies \cite[Prop.\ 18.28]{schmudMomentBook}.

\begin{thm}\label{charattained}
Suppose  $\cS_\sA$ is closed and pointed (i.e., line-free). Let $s\in\cS_\sA$ and $x\in\cX$. Set $s' := s - \rho_s(x) s_\sA(x)\in\partial\cS_\sA$. The following are equivalent:
\begin{enumerate}[i)]
\item The infimum (\ref{eq:rhoLm2}) for $\kappa_s(x)$  is attained at some function $p\in\pos(\sA)$.
\item $x\not\in\cV(s')$.
\end{enumerate}
\end{thm}
\begin{proof}
By Proposition \ref{prop:onep} there is a $p\in\pos(\sA)$ such that $\cV(s') = \cZ(p)$. By  scaling  we can assume $p(x) = 1$. By Proposition \ref{prop:kappaGleichRho} we have $\kappa_s(x) = \rho_s(x) < \infty$. Then
\[\text{i)} \gdw p(x) = 1 \und L_s(p) = \kappa_s(x) \gdw p(x)=1 \und L_{s'}(p)=0 \gdw \text{ii)}.\qedhere\]
\end{proof}

Let us retain the assumptions and the notation of Theorem \ref{charattained}. Then, by this theorem, $x\in\cV(s')$ if and only if the infimum (\ref{eq:rhoLm2}) for $\kappa_s(x)$  is {\it not} a minimum.
Further, if $\cV(s') = \cW(s')$, then the infimum (\ref{eq:rhoLm2}) is attained. Hence the case that the infimum  (\ref{eq:rhoLm2}) is not attained appears only when  $\cV(s') \neq \cW(s')$. That is, $\kappa_s(x)$ is a minimum for all $x\in\cX$ and $s\in\cS_\sA$ if and only if the moment cone $\cS_\sA$ is perfect (according to Definition \ref{dfn:perfect}).

%

From the definition of $s'$ it is clear that $x\not\in\cW(s')$. Hence, each  example  $s\in \cS_\sA$ such that the infimum (\ref{eq:rhoLm2}) is not attained is of the following  form: $s'\in\cS_\sA$ with $\cV(s')\neq\cW(s')$ and $s = s' + c\cdot s_\sA(x)$ for $x\in\cV(s')\setminus\cW(s')$ and $c>0$.

Note that the proof of Theorem 24 in \cite{schmud15} contains a gap. The following example shows that this result does not hold without additional assumptions.

\begin{exm}[Example \ref{exm:harris} revisited]\label{exm:countermaxmass}
Let $\sA = \sB_{2,10}$ on $\cX=\pset^2$ and retain the notation of Example \ref{exm:harris}. Recall that $z_i, i=1,\dots,30,$ are the projective zeros of the Harris polynomial. Let $x_i=z_i$ and $c_i>0$ for $i=1,\dots,23$ and $x=x_j$ for one   $j\in\{24,\dots,30\}$ and $c>0$. Set $$s := \sum_{i=1}^{23} c_i\cdot s_\sA(x_i) + c\cdot s_\sA(x) = s' + c\cdot s_\sA(x).$$
 As shown in Example \ref{exm:harris},  $\cV(s')=\{z_1,\dots,z_{30}\}$ and  $\cW(s')=\{z_1,\dots,z_{23}\}$. 

Then the infimum  (\ref{eq:rhoLm2}) for  $\kappa_s(x) = c=\rho_s(x) $ is not attained. Indeed, if the infimum  (\ref{eq:rhoLm2}) would be  attained at some $p\in\pos(\sA)$ with $p(x)=1$, then $$c = L_s(p) = L_{s'}(p) + L_{c\cdot s_\sA(x)}(p) = L_{s'}(p) + c\cdot p(x) = L_{s'}(p) + c,$$ so that $L_{s'}(p)=0$. Hence $p$ is a multiple of the Harris polynomial and therefore, $p(x)=0$, a contradiction. 
\end{exm}

\section{Open Problems}
\label{sec:open}

The final section is devoted to a list of open problems which are related to the topics treated in this paper.

As shown in Corollary \ref{cor:a10b10}, the cones $\cS_{\sA_{n,10}}$ and $\cS_{\sB_{n,10}}$ for $n\geq 2$ are not perfect. It is likely to expect that this holds for polynomials of higher degrees as well.

\begin{open}\label{open:perfect}
Are the cones $\cS_{\sA_{n,2k}}$ and $\cS_{\sB_{n,2k}}$ perfect for $k\in \nset$?
\end{open}

In Section \ref{sec:face}, the zero sets of the polynomials $\fp$ in (\ref{eq:p}) and $\fq$ in (\ref{eq:q}) played a crucial role and a number of inequalities  (\ref{eq:trends}) were stated for the pairs in table \ref{tab:Ds}. This leads to the following  problems.

\begin{open}\label{open:exposedDims}
Do  the inequalities  in (\ref{eq:trends}) holds for arbitrary pairs $(n,d)$?
\end{open}

\begin{open}\label{open:limits}
Do the limits in (\ref{eq:limitsW}) and (\ref{eq:limitsZ}) exist? If yes, what are these limits?
\end{open}

In Section \ref{sec:cara}, the polynomials $\fp$ and $\fq$ were used to derive bounds for Carath\'eodory numbers.

\begin{open}\label{open:otherPolynomials}
Can the lower bounds of the Carath\'eodory numbers $\cat_\sA$ in table \ref{tab:Ds} be (significantly) improved by using other nonnegative polynomials with finitely many zeros  than $\fp$  and $\fq$? What happens then with the limits in (\ref{eq:limitsW}) and (\ref{eq:limitsZ})?
\end{open}

In Section \ref{sec:inner} we investigated the inner structure of the moment cone $\cS_\sA$. The first question comes from the definition of  regular/singular moment sequences.

\begin{open}\label{open:regular}
Do the regularity/singularity notions in Definition \ref{dfn:regular} depend on $k$? Is it possible that a moment sequence is regular for  $k\in\nset$, but singular for some $k' > k$?
\end{open}

There are two (independent) distinguished properties of ``nice'' behavior of sets $\sA$. This first  is that ``$\cat_\sA=\cN_\sA$'', while the second is stated as (\ref{eq:interregular}): ``$s\in\inter\cS_\sA \gdw s$ is regular''. As shown in Section \ref{sec:inner}, both properties are valid for $\sA = \{1,x,\dots,x^d\}$ on $\rset$ and $\sA = \{x^{d_1},\dots,x^{d_m}\}$ on $(0,\infty)$.

\begin{open}\label{open:internal}
Are there other finite sets $\sA$ of polynomials for which $\cat_\sA = \cN_\sA$ and (\ref{eq:interregular}) hold? Are there other useful properties to distinguish "nice" moment problems?
\end{open}

Using the Harris polynomial we  constructed in Example \ref{exm:harris} a moment sequence $s$ for $\sA = \sB_{2,10}$ on $\cX=\pset^2$ for which $\cV(s)\neq \cW(s)$. Since $\cW(s)=\cV_C(s)$, this means that $\cV_1(s)$ is not the core variety. This suggests the following problem, see also teh discussion after the proof of Theorem \ref{thm:WsZps}.

\begin{open}
Let $\sA = \sB_{n,2k}$ on $\cX=\pset^n$ for $k,n\in \nset$. Given $m\in \nset$, does there exist a moment sequence $s$ such that $\cV_{m-1}(s)\neq \cV_C(s)$  and $\cV_m(s)=\cV_C(s)$?
\end{open}

While  $\rset[x]$ has the Pythagoras number 2, univariate polynomials with gaps can have arbitrary large Pythagoras numbers (see Example \ref{exm:pythlargegaps}).

\begin{open}\label{open:pythagoras}
What is the Pythagoras number of
\[\sA = \{x^{d_1},\dots,x^{d_m},x^{d_m+1},x^{d_m+2},\dots\} \qquad\text{for}\qquad d_1 < d_2 <\dots < d_m,\ d_i\in\nset_0,\]
i.e., finitely many $x^{d}$ in $\rset[x]$ are removed? What happens when infinitely many $x^d$ are removed? Which numbers can be obtained as Pythagoras numbers in this manner?
\end{open}

\section*{Acknowledgment}

P.dD.\ thanks the Deutsche Forschungsgemeinschaft for support from the grand SCHM1009/6-1.


\providecommand{\bysame}{\leavevmode\hbox to3em{\hrulefill}\thinspace}
\providecommand{\MR}{\relax\ifhmode\unskip\space\fi MR }
\providecommand{\MRhref}[2]{%
  \href{http://www.ams.org/mathscinet-getitem?mr=#1}{#2}
}
\providecommand{\href}[2]{#2}

\end{document}